\RequirePackage{rotating}
\documentclass[11pt,a4paper]{amsart}
\usepackage[DIV14,paper=a4,BCOR14mm,headinclude]{typearea}
\usepackage[british]{babel}
\usepackage{amsmath, amssymb}
\usepackage{graphicx}
\usepackage{color}
\usepackage{tikz}
\usepackage{pgfplots}
\usepackage{pgfplotstable}
\usepackage{rotating}
\usepackage{lscape}
\usepackage{enumitem}
\usepackage{hyperref}
\usepackage{fp-eval}
\usepackage{esint}
\usepackage{subfigure}
\usepackage{caption}
\usepackage{todonotes}
\usepackage{mathabx}
\usepackage{booktabs}
\usepackage{ngerman}
\definecolor{darkblue}{rgb}{0,0,.7}
\hypersetup{colorlinks=true,linkcolor=darkblue,citecolor=darkblue}

\newlist{alphenum}{enumerate}{1}
\setlist[alphenum]{fullwidth,label={(\alph*)}}

\allowdisplaybreaks[0]
\theoremstyle{definition}
\newtheorem{theorem}{Theorem}[section]

\newtheorem{remark}[theorem]{Remark}
\newtheorem{lemma}[theorem]{Lemma}
\newtheorem{definition}[theorem]{Definition}

\numberwithin{figure}{section}
\numberwithin{table}{section}
\numberwithin{equation}{section}

\newcommand{\epstensor}{\boldsymbol{\epsilon}}

\newcommand{\bzero}{\mathbf{0}}

\newcommand{\ba}{\mathbf{a}}
\newcommand{\bfe}{\mathbf{f}}
\newcommand{\bg}{\mathbf{g}}
\newcommand{\bn}{\mathbf{n}}

\newcommand{\bu}{\mathbf{u}}
\newcommand{\bv}{\mathbf{v}}
\newcommand{\bw}{\mathbf{w}}
\newcommand{\bx}{\mathbf{x}}

\newcommand{\bH}{\mathbf{H}}
\newcommand{\bL}{\mathbf{L}}
\newcommand{\bV}{\mathbf{V}}

\newcommand{\vecb}[1]{\textbf{#1}}

\begin{document}
\date{\today}

%% TITLE %%%%%%%%%%%%%%%%%%%%%%%%%%%%%%%%%%%%%%%%%%%%%%%
\title[A gradient-robust scheme]{A gradient-robust well-balanced
scheme for the compressible isothermal Stokes problem}

%% AUTHORS %%%%%%%%%%%%%%%%%%%%%%%%%%%%%%%%%%%%%%%%%%%%%
\author{M.~Akbas, T.~Gallou\"et, A.~Ga{\ss}mann, A.~Linke, C.~Merdon}
\maketitle

%%%%%%%%%%%%%%%%
\begin{abstract}
A novel notion for constructing a well-balanced scheme
--- a gradient-robust scheme --- is introduced
and a showcase application for a steady compressible, isothermal
Stokes equations is presented.
Gradient-robustness means that
arbitrary gradient fields
in the momentum balance are
well-balanced by the discrete pressure gradient ---
if there is enough mass in the system to compensate the force.
The scheme is asymptotic-preserving in the sense
that it degenerates
for low Mach numbers to a recent inf-sup stable and
pressure-robust discretization for the incompressible Stokes
equations. The convergence
of the coupled FEM-FVM scheme for the nonlinear, isothermal Stokes
equations is proved by compactness arguments.
Numerical examples illustrate the numerical analysis, and show
that the novel approach
can lead to a dramatically increased accuracy in
nearly-hydrostatic low Mach number
flows. Numerical examples also suggest that a straight-forward extension to
 barotropic situations with nonlinear equations of state
 is feasible.
\end{abstract}

%%%%%%%%%%%%%%%%%%%%%%
\section{Introduction}
In recent years, novel concepts and discretization approaches for the incompressible Navier-Stokes equations appeared around the so-called \textit{pressure-robustness} property. Such discretizations allow for a priori error estimates
of the discrete velocity that are independent of the pressure and the viscosity parameter that otherwise gives rise to a severe locking phenomenon \cite{JLMNR:sirev,cmame:linke:merdon:2016,
MR3767413,MR3824769,MR3780790,MR3767413,MR3743746,MR3826676,MR3833698,
MR3511719, MR2304270, MR3833698,
MR818790, MR2114637}
demonstrated in several benchmark examples \cite{JLMNR:sirev,cmame:linke:merdon:2016,MR3481034,MR3824769, gauger:linke:schroeder:2019}
even in coupled problems to simulate electrolyte flows \cite{WIASpreprint2525}.
Surprisingly, an astonishingly simple
modification  that
manipulates locally the
velocity test functions in the right-hand side (and the material derivative if present)
in order to restore the orthogonality between discretely divergence-free functions and gradients renders any classical inf-sup stable finite element method a pressure-robust method \cite{MR3133522,MR3824769,MR3743746,MR3656505}.

In this contribution we apply
this modification to a provably convergent discretization
of the compressible Stokes equations inspired by 
\cite{gal-09-conv,MR3749377}. The proposed modification
does not compromise the convergence analysis, but improves the
accuracy in nearly hydrostatic low Mach number flows.
Thereby, a novel notion for a certain
class of well-balanced schemes
--- {\em gradient-robust} schemes ---
for vector-valued partial differential equations
like the compressible Euler, the compressible Navier--Stokes
or the shallow-water equations
is introduced.

The notion \emph{gradient-robust} wants to emphasize that
the accuracy of these schemes does not suffer from the
appearance of dominant gradient fields in the momentum balance,
leading to an accurate, implicitly defined discrete vorticity equation
\cite{JLMNR:sirev}.
Indeed, several schemes for several
different \emph{vector PDEs} can be classified
as \emph{gradient-robust}, e.g., see
 \cite{cotter:thuburn:2014,
natale:cotter:2018, guzman:et:al:2017}. In the meteorology
community such schemes have been introduced
by Cotter and Thuburn and their well-balanced property
has been explained by \emph{exact sequences} in the
setting of finite element exterior calculus \cite{cotter:thuburn:2014}.
The proposed explanation for
the well-balanced property of these schemes
below is complementary,
but sets a different focus. Emphasizing the
importance of an 
$\bL^2$ orthogonality of certain (discretely) divergence
vector fields against (arbitrary) gradient fields for accuracy
reason allows to build novel \emph{gradient-robust} schemes.

In order to compare {\em gradient-robustness}
with classical well-balanced schemes \cite{XingShu2013, BottaKlein2004},
we regard the
{\em momentum balance of isothermal hydrostatics}
\begin{equation} \label{eq:example}
  (\rho \vecb{u})_t +
   \nabla \cdot \left ( \rho \vecb{u} \otimes \vecb{u}
     \right ) + \nabla p = -\rho \nabla \phi,
\end{equation}
for the compressible and the incompressible Euler equations
on a bounded polyhedral Lip\-schitz domain $\mathcal{D}$.
Here, the potential $\phi$ is assumed to depend on the space variable
$\vecb{x}$ only, i.e., we search for a steady density $\rho(\vecb{x})$
and pressure solutions $p(\vecb{x})$ fulfilling \eqref{eq:example}
with $\vecb{u} \equiv \vecb{0}$. The goal of the comparison
is to understand
the necessary properties for a space discretization that
discretely preserves incompressible or compressible {\em hydrostatics}.

In the incompressible case, it holds $\rho = \mathrm{const}$
and therefore one concludes for a hydrostatic balance
\begin{equation*}
 \nabla p = -\rho \nabla \phi = \nabla \left (-\rho \phi \right ),
\end{equation*}
i.e., the hydrostatic
pressure is given (up to a constant) by
$p = -\rho \phi + \mathrm{const}$. A consistent discretization
of incompressible hydrostatics requires to balance
{\em arbitrary gradient fields $\nabla (- \rho \phi)$} by the discrete pressure gradient,
i.e., a consistent, pressure-robust discretization
possesses an appropriately defined discrete Helmholtz projector
$\mathbb{P}_h$
--- an $\vecb{L}^2$ projector onto discretely divergence-free vector fields
--- whose kernel contains {\em arbitrary (!) gradient fields}
in $\vecb{L}^2$, i.e., it holds
\begin{equation}
  \mathbb{P}_h(-\rho \nabla \phi) = \vecb{0},
\end{equation}
for all $\phi \in \vecb{H}^1(\mathcal{D})$.
Pressure-robust schemes for incompressible flows achieve this goal
by exploiting the $\vecb{L}^2$-orthogonality
of vector-valued, {\em divergence-free} $H(\mathrm{div})$-conforming finite element test functions
(with vanishing normal component at the boundary)
against arbitrary gradient fields.
Thus, pressure-robust schemes can be constructed
on general unstructured grids. Note that most
classical finite element, finite volume and Discontinuous
Galerkin schemes are {\em not pressure-robust}.
The kernel of their discrete Helmholtz projectors
contains only a subspace of {\em discrete pressure gradients} \cite{cmame:linke:merdon:2016,JLMNR:sirev}.

Assuming for the isothermal ($T=\mathrm{const}$) compressible case an ideal gas law $p = \rho \mathcal{R} T$, the hydrostatic balance
is given by
\begin{equation*}
 \nabla p = -\rho \nabla \phi  \qquad \Leftrightarrow \qquad
 \nabla \rho = \rho \nabla \left (-\frac{\phi}{\mathcal{R} T} \right ),
\end{equation*}
which can be explicitely integrated, leading to
$$
  \rho = \rho_0 \exp \! \left ( \frac{-\phi}{\mathcal{R} T} \right ),
  \qquad 
  p = \rho_0 \mathcal{R} T \exp \! \left ( \frac{-\phi}{\mathcal{R} T} \right ).
$$
Exploiting this {\em explicit solution}, one confirms the identity
$$
  \mathcal{R} T \exp \! \left ( \frac{\phi}{\mathcal{R} T} \right )
    \nabla \left ( \exp \! \left ( \frac{-\phi}{\mathcal{R} T} \right)
     \right ) = -\nabla \phi
$$
and --- as an example ---
the classical well-balanced scheme \cite{XingShu2013}
is based on a space discretization
of the right hand side term in the form
\begin{equation} \label{eq:class:compr:wbs}
 -\rho \nabla \phi = 
   \rho_0 \mathcal{R} T  \, \nabla \left ( \exp \! \left ( \frac{-\phi}{\mathcal{R} T} \right)
     \right ).
\end{equation}

The key ideas of a {\em gradient-robust} well-balanced
scheme for compressible flows, which is based on the notion of pressure-robustness for incompressible flows, rely on the following observations:
\begin{enumerate}[fullwidth]
\item[(a)] A hydrostatic balance
$$
  \nabla p = -\rho \nabla \phi
$$
is only possible if $\rho \nabla \phi$ is a gradient-field,
which is only a-priori clear in the incompressible case $\rho=\mathrm{const}$; thus it seems to be plausible for variable $\rho$ 
that a more accurate treatment
of gradient forces may increase the overall accuracy of the scheme
in a nearly-hydrostatic situation.
\item[(b)] Actually, we demonstrate in this contribution that
it is possible to construct {\em gradient-robust schemes} on arbitrary unstructured grids, which allow a well-balanced property
of the form
$$
 \nabla p = -\nabla \psi
$$
for arbitrary gradient fields $\nabla \psi \in \vecb{L}^2(\mathcal{D})$ --- if there is enough mass in the system to compensate the gradient force $-\nabla \psi$.
\item[(c)] The velocity field of nearly-hydrostatic flows is
of {\em low Mach number} type. Thus,
its non-divergence-free part is small, i.e., of order
$\mathcal{O}(\frac{1}{\mathrm{Ma}^2})$, and
for an accurate treatment in nearly-hydrostatic,
low Mach number flows it suffices
to achieve that the {\em divergence-free part} of the velocity
field vanishes in the hydrostatic limit case.
Exactly this is achieved in this contribution for
the barotropic compressible Stokes equations, since it is
shown that the divergence-free part of the velocity fulfills
the {\em incompressible Stokes equations} --- where
(incompressible)
{\em pressure-robustness} can be exploited.
\end{enumerate}

Although, this paper develops an appropriate space discretization for the
rather simple compressible Stokes equations, it is a first step to
develop and analyze {\em gradient-robust} schemes for the full compressible Navier--Stokes equations. Nevertheless, even this
simple physical problem is highly relevant in atmospheric and oceanic modeling where the correct representation of the hydrostatic balance
between pressure gradient and gravity is of vital importance in
{\em stably stratified airflows over a topography}.
It has been early recognized that discretization errors especially in terrain-following coordinates --- leading to rather structured grids ---
can become large and deteriorate the accuracy of the numerical solutions. Such errors are especially severe, if a resting fluid is located over steep terrain \cite{zaengl2004}. Several attempts were made over the years to ameliorate the simulations. Proposed methods are the increase of the order of accuracy \cite{ENGWIRDA20171,Zaengl2012,Mahrer1984}, the improvement of the lower boundary condition \cite{Gassmann2004}, the usage of cut cells or step mountain coordinates instead of the terrain following coordinates \cite{Steppeler2019,SteppelerMinotte2002,ShawWeller2016}, the reduction of the steepness of the slopes \cite{Schaer2002}, the damping of error induced noise \cite{Zaengl2002}, covariant formulations of the pressure gradient term \cite{Lin1997,LiLiWang2016}, curl-free pressure gradient formulations \cite{WellerShahrokhi2014}, and energy conserving schemes (discrete Poisson brackets) that hinder the spurious increase of kinetic energy in the perturbations \cite{Gassmann2013}. These references are just to be thought
to reflect the importance of the problem in applications.

The rest of the paper is structured as follows. Section~\ref{sec:cStokes} introduces
the steady, compressible isothermal Stokes equations, which
serve as a model problem.
Section~\ref{sec:discrete_scheme} explains our discretization, in particular the
finite-volume scheme for the continuity equation and the finite element scheme for the momentum equation with the gradient-robust right-hand side modification.
Section~\ref{sec:gradient_robustness} motivates and discusses the new gradient-robustness property and links it to pressure-robustness in the incompressible setting or the
well-balanced known from shallow water equations.
Section~\ref{sec:existence} proves the existence of a discrete solution by standard compactness
arguments, while Section~\ref{sec:convergence} proves the convergence of a series
of discrete solutions to a weak solution of the continuous system.
In Section~\ref{sec:numerics} the theoretical findings are validated by appropriate numerical benchmarks.

%%%%%%%%%%%%%%%%%%%%%%%%%%%%%%%%%%%%%%%%%%%%%%%%%%%%%%%%%%%%%%%%%%%
\section{A model problem: the steady compressible isothermal Stokes equations} \label{sec:cStokes}
The isothermal compressible Stokes problem seeks
for $(\bfe, \bg) \in \bL^2(\Omega) \times \bL^\infty(\Omega)$
some velocity field \(\vecb{u}\), pressure \(p\) and non-negative density \(\varrho \geq 0\) with \(\int_\Omega \varrho \, dx = M\) such that
\begin{equation} \label{eq:compr:stokes:problem}
\begin{split}
  -\nabla \cdot \boldsymbol{\sigma} + \nabla p & = \vecb{f} + \varrho \vecb{g},\\
  \mathrm{div}(\varrho \bu) & = 0\\
  p & = \varphi(\rho) := c \rho,
\end{split}
\end{equation}
where friction is modeled as in linear elasticity
by
\begin{equation}
  \boldsymbol{\sigma} =  2 \mu \epstensor(\bu) + \lambda (\nabla \cdot \bu) \boldsymbol{I},
\end{equation}
with
$\epstensor(\bu) := \frac{1}{2} (\nabla \bu + (\nabla \bu)^T)$, $\mu \in \mathbb{R}^+$,
$\lambda \in \mathbb{R}$ with $\lambda \geq \underline{\lambda}  > -2 \mu$, compare e.g.\ with \cite{feireisl:2010},
and where the equation of state function \(\varphi:= c \rho\)
with $c > 0$ is prescribed in addition to homogeneous
Dirichlet velocity boundary conditions to close the system. Note, that the constant \(c\) may model (in a dimensionless setting) the squared inverse of the Mach number.

The compressible Stokes problem is thus a nonlinear problem
and can be written in the following weak form \cite{MR2600538}:
search for
 \((\vecb{u},p,\varrho) \in \vecb{H}^1_0(\Omega) \times L^2(\Omega) \times L^{2}(\Omega)\) with
\begin{align}\label{eqn:compressible_weak}
  a_1(\vecb{u},\vecb{v}) + a_2(\vecb{u},\vecb{v}) + b(p,\vecb{v}) & = F(\vecb{v}) + G(\varrho, \vecb{v}) && \text{for all } \vecb{v} \in \vecb{H}^1_0(\Omega),\\\nonumber
  c(\varrho,u,\phi) & = 0 && \text{for all } \phi \in W^{1,\infty}(\Omega),
\end{align}
where the multilinear forms used above read as
\begin{align*}
  a_1(\vecb{u},\vecb{v}) & := 2 \mu \int_\Omega \epstensor(\bu) : \epstensor(\bv) \, dx,
  & a_2(\vecb{u},\vecb{v}) & := \lambda \int_\Omega \mathrm{div}(\vecb{u}) \mathrm{div}(\vecb{v}),\\
  b(p,\vecb{u}) & := -\int_\Omega p \, \mathrm{div}(\vecb{u}) \, dx,
  & c(\varrho, \bu,\phi) & := \int_\Omega \varrho \bu \cdot \nabla \phi \, \, dx,\\
  F(\vecb{v}) & := \int_\Omega \vecb{f} \cdot \vecb{v} \, dx,
  & G(\varrho,\vecb{v}) & := \int_\Omega \varrho \vecb{g} \cdot \vecb{v} \, dx.
\end{align*}

Qualitative properties of the velocity solution
$\bu$ can be investigated by introducing the spaces
\begin{align}
  \bV^0 & = \{ \bv \in \bH^1_0(\Omega)  : \nabla \cdot \bv = 0\} \nonumber\\
 \bV^\perp & = \{ \bv \in \bH^1_0(\Omega) : (\epstensor(\bv), \epstensor(\bw)) = 0
  \quad \text{for all $\bw \in \bV^0$}
   \}  \label{eqn:def_Vperp}
\end{align}
and with the help of the orthogonal splitting
--- in the scalar product $(\epstensor(\bullet), \epstensor(\bullet))$ ---
\begin{equation*}
  \bu = \bu^0 + \bu^\perp
\end{equation*}
with $\bu^0 \in \bV^0$ and $\bu^\perp \in \bV^\perp$.
Testing the equation with an arbitrary $\bv^0 \in \bV^0$
one recognizes that it holds
\begin{equation} \label{eq:u0}
 2\mu (\epstensor(\bu^0), \epstensor(\bv^0))
   = (\vecb{f} + \rho \vecb{g}, \bv^0)
\end{equation}
for all $\bv^0 \in \bV^0$. Thus, for fixed $\rho$ or for
$\vecb{g} = \vecb{0}$ the
\emph{divergence-free part $\bu^0$ }of the solution $\bu$
fulfills a \emph{linear incompressible Stokes equations}.
Moreover, introducing the space
\begin{equation}
  \bL^2_\sigma = \{ \bv \in \bL^2(\Omega) :  (\bv, \nabla \phi) = 0 \, \text{for all $\phi \in H^1(\Omega)$} \},
\end{equation}
one obtains the orthogonal decomposition
\begin{equation} \label{eq:helmholtz:hodge:decomp}
  \bL^2(\Omega) = \bL^2_\sigma \oplus_{\bL^2}
   \{ \nabla \phi : \phi \in H^1(\Omega) \}
\end{equation}
and $\bL^2_\sigma$ represents the space of \emph{weakly divergence-free} $\bL^2$ vector fields with vanishing normal component
at the boundary \cite{JLMNR:sirev}.
Exploiting the Helmholtz--Hodge decomposition
\eqref{eq:helmholtz:hodge:decomp}, one can introduce
the $\bL^2$-orthogonal Helmholtz--Hodge projector
$\mathbb{P}: \bL^2(\Omega) \to \bL^2_\sigma$ 
of a vector field $\vecb{f} \in \bL^2(\Omega)$ by
$$
 (\mathbb{P}(\vecb{f}), \bw) = (\vecb{f}, \bw) \qquad \text{for
 all $\bw \in \bL^2_\sigma$},
$$
see  \cite{JLMNR:sirev}. Then, due to the orthogonal decomposition \eqref{eq:helmholtz:hodge:decomp} one obtains:
\begin{lemma} \label{lem:helm:consistency}
For all $\phi \in H^1(\Omega)$ it holds
$$
  \mathbb{P}(\nabla \phi) = \bzero.
$$
\end{lemma}
%%%%%%%%%%%

Using the concept of the Helmholtz--Hodge projector, one
can refine \eqref{eq:u0} to observe
\begin{equation} \label{eq:u0:helm}
\mu (\epstensor(\bu^0), \epstensor(\bv^0))
   = (\mathbb{P}(\vecb{f} + \rho \vecb{g}), \bv^0)
\end{equation}
for all $\bv^0 \in \bV^0$, i.e., the divergence-free
part $\bu^0$ does not depend on the entire data
$\vecb{f} + \rho \vecb{g}$, but only on its divergence-free part $\mathbb{P}(\vecb{f} + \rho \vecb{g})$.

\section{Well-balanced Bernardi--Raugel finite element - finite volume method}\label{sec:discrete_scheme}
%This section explains our modified Bernardi--Raugel finite element method to discretize the compressible Stokes system. 
%
The proposed discretization is based on the finite element-finite volume scheme of \cite{gal-09-conv}. Here, the continuity equation is discretized by some finite volume technique that ensures the non-negativity and mass constraints of the piecewise-constant discrete density \(\varrho_h\).

For the velocity the classical $\bH^1$-conforming
Bernardi--Raugel finite element method is employed ---
instead of the \emph{nonconforming Crouzeix--Raviart} element
used in \cite{gal-09-conv,MR3749377}.
This has several advantages: First, the conforming method
is cheaper in terms of the number of degrees of freedom. Second, it easily
allows for the use of the stress tensor $\boldsymbol{\sigma}$, whereas the
Crouzeix--Raviart element does not fulfill a \emph{discrete
Korn inequality}. Third, the conforming setting makes
some of the compactness arguments easier, in order to prove
convergence to a weak solution of this nonlinear problem,
without resorting to additional stability terms.

However, the main important difference to the scheme
\cite{gal-09-conv} is a modified discretization
of the right-hand side $\vecb{f} + \rho \vecb{g}$ that
delivers more accurate results in nearly hydrostatic
situations.
The modification is inspired by
certain \emph{pressure-robust schemes} for the incompressible Stokes equations, see e.g.\ \cite{JLMNR:sirev,cmame:linke:merdon:2016}.
Fundamental is an appropriate discrete equivalent of \eqref{eq:u0:helm},
where the discretely divergence-free part $\bu^0_h$
of the discrete solution $\bu_h$ does only depend
on the \emph{continuous} Helmholtz--Hodge projector
$\mathbb{P}(\vecb{f} + \rho_h \vecb{g})$.

\subsection{Notation}
Consider a shape-regular triangulation \(\mathcal{T}\) with nodes \(\mathcal{N}\) and
faces \(\mathcal{F}\). The subset \(\mathcal{F}(\Omega)\) denotes the interior faces of the triangulation.
The set \(P_k(T)\) consists of all scalar-valued polynomials of total degree \(k\) on the simplex \(T \in \mathcal{T}\). Moreover, the set of piecewise polynomials is denoted by
\begin{align*}
  P_k(\mathcal{T}) := \lbrace v_h \in L^2(\Omega) : v_h|_T \in P_k(T) \text{ for all } T \in \mathcal{T} \rbrace.
\end{align*}
Vector-valued quantities or functions are addressed by bold letters.

\subsection{Finite Element Method and a divergence-free
reconstruction operator}
The numerical discretization employs the Bernardi--Raugel finite element spaces
\begin{align*}
  \vecb{V}_h := \left(\vecb{P}_1(\mathcal{T}) \cap \vecb{H}^1_0(\Omega)\right) \oplus \mathcal{B}(\mathcal{F}(\Omega)) 
  \quad \text{and} \quad 
  Q_h := P_0(\mathcal{T}) \cap L^2_0(\Omega),
\end{align*}
where \(\mathcal{B}(\mathcal{F}(\Omega)) \) denotes the normal-weighted face bubbles, i.e.
\begin{align*}
  \mathcal{B}(\mathcal{F}(\Omega)) 
  := \lbrace b_F \vecb{n}_F : F \in \mathcal{F}(\Omega) \rbrace.
\end{align*}
For $d=2$, \(b_F\) is the standard quadratic face bubble on the face \(F \in \mathcal{F}\). For $d=3$, the
corresponding standard face bubble is cubic.
The $L^2$ projection in the discrete pressure $Q_h$ will be
denoted in the following by $\pi_0$.

Then, the \emph{discrete divergence} operator 
$\mathrm{div}_h: \bV_h \to Q_h$ of the Bernardi--Raugel element is denoted by
\begin{equation}
  \mathrm{div}_h(\bv_h) := \pi_0 (\mathrm{div} \, \bv_h).
\end{equation}
 Note that the Bernardi--Raugel element
is \emph{discretely inf-sup stable} on shape-regular meshes
\cite{GirRav-nse}. The space of \emph{discretely divergence-free} vector fields will be denoted as
\begin{equation}
  \bV^0_h = \{  \bv_h \in \bV_h : \mathrm{div}_h \bv_h = 0 \}.
\end{equation}
Due to the general theory of mixed finite element spaces
\cite{GirRav-nse},
the space of discretely divergence-free Bernardi--Raugel
functions has optimal approximation properties versus
$\bV^0$. More precisely, it holds for all $\bv^0 \in \bV^0$
that
\begin{equation}
  \inf_{\bv^0_h \in \bV^0_h} \| \nabla(\bv^0 - \bv^0_h) \| \leq (1 + C_F) \inf_{\bv_h \in \bV_h} \| \nabla (\bv^0 - \bv_h) \|,
\end{equation}
where $C_F$ denotes the (uniformly bounded)
stability constant of the corresponding
\emph{Fortin operator} \cite{GirRav-nse}
of the Bernardi--Raugel element.

The gradient-robust modification of the Bernardi--Raugel finite element method employs
a reconstruction operator \(\Pi\) in the right-hand side functionals, which maps
\emph{discretely divergence-free functions} onto
\emph{weakly divergence-free} ones
in the sense of $\bL^2_\sigma$ \cite{cmame:linke:merdon:2016}.
For the Bernardi--Raugel finite element method, this can be
ensured by standard interpolators into
either the Raviart--Thomas $\mathrm{RT}_0$
or the Brezzi--Douglas--Marini \(\mathrm{BDM}_1\)
finite element spaces \cite{brezzi:fortin}.
Here, we employ the Brezzi--Douglas--Marini standard \(\mathrm{BDM}_1\) standard interpolator defined by
\begin{align*}
  \int_F q_h (\Pi \vecb{v}_h - \vecb{v}_h) \cdot \vecb{n}_F \, ds & =0 \quad \text{for all } q_h \in P_1(F) \text{ and } F \in \mathcal{F}.
\end{align*}
Also note, that \(\Pi(\vecb{v}_h) = \vecb{v}_h\) whenever \(\vecb{v}_h \in \vecb{P}_1(\mathcal{T}) \cap \vecb{H}^1_0\), hence only face bubbles are modified by the reconstruction operator (and their reconstruction equals their $\mathrm{RT}_0\) standard interpolation).
%%%%%%%%%%%%%%
\begin{remark}
In order to give an impression how the proposed space
discretization can actually be implemented, we describe
the discretization variant with the Raviart--Thomas
standard interpolator in detail, although we will not
use this slightly less accurate variant in our numerical
experiments.

The Raviart--Thomas standard interpolator can be
elementwise defined in an explicit way by
\begin{equation}
 \Pi^{\mathrm{RT}_0}(\bv_h)_{| T}= \ba_T + \frac{c_T}{d} \left ( \bx - \bx_T \right ),
\end{equation}
where $\bx_T$ denotes the barycenter of the element $T$,
$c_T$ denotes the elementwise divergence computable by
\begin{equation} \label{eq:rt0:div}
  c_T := \frac{1}{|T|} \sum_{F \in \mathcal{F}(T)}
   \int_F \bv_h \cdot \bn_F \, \mathit{dS}
\end{equation}
and $\ba_T$ denotes the \emph{average velocity} computable by
$$
  \ba_T := \frac{1}{|T|} \sum_{F \in \mathcal{F}(T)}
    \left ( \int_F \bv_h \cdot \bn_F \, \mathit{dS} \right )
    (\bx_F - \bx_T),
$$
where $\bx_F$ denotes a face barycenter of the face $F$.
\end{remark}
 The following lemma collects some more important properties.

\begin{lemma}[Properties of \(\Pi\)]
  It holds, for all \(\vecb{v} \in \vecb{H}^1_0(\Omega)\),
  \begin{align}
    \mathrm{div}(\Pi \vecb{v}) & = \mathrm{div}_h
    \bv, \label{eq:pi:prop:1} \\
    \| \vecb{u} - \Pi \vecb{v} \|_{L^2(T)} & \leq  h_T \|\nabla \vecb{v} \|_{L^2(T)} \quad \text{for all } T \in \mathcal{T},
    \label{eq:pi:prop:2} \\
    \int_\Omega \nabla \phi \cdot (\Pi \vecb{v}) \, dx & =  - \int_\Omega \phi \, \mathrm{div}_h  \bv \, dx \quad \text{for all } \phi \in H^1(\Omega). \label{eq:pi:prop:3}
  \end{align}
\end{lemma}
\begin{proof}
  The properties \eqref{eq:pi:prop:1} and \eqref{eq:pi:prop:2} follow from the properties of the standard interpolation into the spaces \(\mathrm{BDM}_1\)
  (and $\mathrm{RT}_0$), see e.g.\ \cite{brezzi:fortin}.
  %Actually, in the Raviart--Thomas case \eqref{eq:rt0:div}
  %tells immediately that \eqref{eq:pi:prop:1} holds, since
  %\emph{elementwise
  %divergence-free} vector fields will lead to
  %$c_T = 0$ according to the \emph{divergence theorem}.
   Property \eqref{eq:pi:prop:3} follows from an integration by parts and property \eqref{eq:pi:prop:1}.
\end{proof}
%%%%%%%%%%%%%%
\begin{remark}
The cornerstone of the novel gradient-robust scheme
is given by the following statement:
for all \emph{discretely-divergence-free} Bernardi--Raugel
vector fields $\bv^0_h \in \bV^0_h$ and all $\phi \in H^1(\Omega)$ it holds
$$
  \int_\Omega \nabla \phi \cdot (\Pi \bv^0_h) \, dx
  =  - \int_\Omega \phi \, \mathrm{div}_h  \bv^0_h \, dx
  = 0,
$$
i.e., the reconstruction operator $\Pi$
enables to repair the $\bL^2$ orthogonality
of \emph{discretely divergence-free} vector fields
and arbitrary gradient fields $\nabla \phi$.
\end{remark}
%%%%%%%%%%%%

\subsection{Coupling to finite volume upwind discretization of continuity equation}\label{subsection:pseudo:time:stepping}

This finite element scheme is coupled to a finite volume discretization for the continuity equation. Altogether, our discretization seeks some \((\vecb{u}_h,p_h,\varrho_h) \in \vecb{V}_h \times Q_h \times Q_h\) such that
\begin{align}
  a_1(\vecb{u}_h,\vecb{v}_h) + a_{2}(\Pi \vecb{u}_h,\Pi \vecb{v}_h) + b(p_h,\vecb{v}_h)& = F(\Pi \vecb{v}_h) + G(\varrho_h, \Pi \vecb{v}_h)  && \text{for all } \vecb{v}_h \in \vecb{V}_h,
  \label{eq:gradient:robust:scheme} \\
    \mathrm{div}_\text{upw}(\varrho_h \vecb{u}_h) & = 0, &&
    \nonumber \\
    p_h & = \varphi(\varrho_h). && \nonumber
\end{align}

The upwind discretization \(\mathrm{div}_\text{upw}(\varrho_h \vecb{u}_h) \in P_0(\mathcal{T})\) of \(\mathrm{div}(\varrho_h \vecb{u}_h)\) is defined
on all $T \in \mathcal{T}$ by
\begin{align*}
  \mathrm{div}_\text{upw}(\varrho_h \vecb{u}_h)|_T
  & :=
  \frac{1}{|T|} \sum_{F \in \mathcal{F}(T) \cap \mathcal{F}(L)} u_{T,F}^+ \varrho_h|_T - u_{T,F}^-\varrho_h|_L \\
 &  = \frac{1}{|T|}
  \sum_{F \in \mathcal{F}(T) \cap \mathcal{F}(L)} \varrho^\text{upw}_F u_{T,F},
\end{align*}
where \(u_{T,F} = \int_F \vecb{u}_h \cdot \vecb{n}_T \, ds\) is the integral over
the face \(F\) in outer normal direction of the simplex \(T\) and \(u_{K,F}^+ \geq 0\)
and \(u_{T,F}^- \geq 0\) is the positive and negative part, respectively. Hence,
\(\varrho^\text{upw}_F := \varrho_h|_T\) if \(u_{T,F} > 0\) and \(\varrho^\text{upw}_F := \varrho_h|_L\)
else for \(F = \partial T \cap \partial L\).

The introduction of the upwind divergence leads to a (singular) matrix
\begin{align}\label{eqn:definition_matrixD}
 \mathrm{div}_\text{upw}(\varrho_h \vecb{u}_h) = 0 \quad \Leftrightarrow \quad D \varrho_h = 0
  \quad \text{where} \quad D_{jk} := \mathrm{div}_\text{upw}(\chi_j \vecb{u}_h)|_{T_k}
\end{align}
  where \(\chi_j\) is the characteristic function of \(T_j \in \mathcal{T}\).
  
\begin{lemma}[Properties of $D$]
  It holds
\begin{align*}
  & (1) & D \text{ is } & \text{weakly diagonal-dominant,i.e. } \\
  && D_{jj} &\geq 0 
  \quad \text{and} \quad
  \sum_k D_{jk} = 0 \quad \text{for all } j=1,\ldots,\mathcal{T},\\
  & (2) & D^T \vecb{1} &= 0,\\
  & (3) & D \vecb{1} & = \mathrm{div}_\text{upw}(\vecb{u}_h) = \pi_0 \mathrm{div}(\vecb{u}_h).
\end{align*}
\end{lemma}  
\begin{proof}
  The proof of (1) and (2) is common for finite volume discretizations
  and follows straightforwardly from the relation
  \(u_{T,F}^\pm = u_{L,F}^\mp\) for \(F \in \mathcal{F}(T) \cap \mathcal{F}(L)\). For the proof of (3) recall \(u_{L,F} = u_{L,F}^+ - u_{L,F}^-\).
\end{proof}

Since \eqref{eq:gradient:robust:scheme} is a nonlinear problem, it has to be solved iteratively and one has to choose a reasonable solution \(\varrho_h\) with \(\mathrm{div}_\text{upw}(\varrho_h \vecb{u}_h) = 0\) that satisfies the non-negativity and mass constraints.
Consider a given approximation \(\vecb{u}_h\) and \(\varrho_h^{n-1}\)
(from a previous fixpoint iterate or an initial solution). To compute a unique update \(\varrho_h^{n}\) of the discrete density that preserves the non-negativity and the integral mean of \(\varrho_h^{n-1}\), we suggest to employ the backward Euler method. Given the (diagonal) \(P_0\) mass matrix \(M\), i.e.
\(M_{jj} := \lvert T_j \rvert\) and some time step \(\tau\), this leads to the linear problem
\begin{align}\label{eqn:discrete_timestep_continuity}
  (M + \tau D) \varrho_h^{n} = M \varrho_h^{n-1}.
\end{align}
Here, \(\varrho_h^{n}\) has to be understood as a column vector with
the elementwise constant values of \(\varrho_h^{n} \in P_0(\mathcal{T})\).

\begin{lemma}[Preservation of non-negativity and integral mean]
\label{lem:mass_conservation}
  It holds
  \begin{align*}
    & (1) & \varrho_h^{n} & \geq 0 \quad \text{if } \varrho_h^{n-1} \geq 0,\\
    & (2) & M (\varrho_h^{n} - \varrho_h^{n-1}) \cdot \vecb{1} & = 0.
  \end{align*}
\end{lemma}  
\begin{proof}
  Since \(M\) is a positive diagonal matrix, \(M + \tau D\) is diagonal-dominant
  and hence an \(M\)-matrix. This implies that the inverse \((M + \tau D)^{-1}\)
  is totally positive and hence preserves the non-negativity of \(\varrho_h^{n-1}\).
  The second property follows from \(D^T \vecb{1} = 0\) used in the identity
  \begin{align*}
    (M \varrho_{h}^{n}) \cdot \vecb{1} 
  = \varrho_{h}^{n} \cdot ((M + \tau D)^T \vecb{1})
     = ((M + \tau D) \varrho_{h}^{n}) \cdot \vecb{1}
     = (M \varrho_{h}^{n-1}) \cdot \vecb{1}.
  \end{align*}
  This concludes the proof.
\end{proof}

The pseudo time-stepping \eqref{eqn:discrete_timestep_continuity} is embedded into the iterative algorithm in Section~\ref{sec:algorithm}.

%%%%%%%%%%%%%%%%%%%%%%%%%%%%%%%%%%%%%%%%%%%%%%%%%%%%%%%%%%%%%%%%%
\section{On gradient-robustness and well-balanced schemes}
\label{sec:gradient_robustness}
In analogy to \eqref{eqn:def_Vperp}, consider the discrete space
$$
  \bV^\perp_h := \{ \bv_h \in \bV_h :
  (\epstensor(\bv_h), \epstensor(\bw^0_h)) = 0
  \, \text{ for all $\bw^0_h \in \bV^0_h$} \}
$$
which allows for the orthogonal splitting $\bV_h := \bV^0_h \oplus \bV^\perp_h$
in the discrete scalar product $(\epstensor(\bullet), \epstensor(\bullet))$.
The main structural property of the gradient-robust
scheme \eqref{eq:gradient:robust:scheme} is now derived by:
\begin{theorem} \label{thm:grad:robust:structure}
Exploiting the splitting $\bu_h = \bu^0_h + \bu^\perp_h$
with $\bu^0_h \in \bV^0_h$ and $\bu^\perp \in \bV^\perp_h$,
the \emph{discretely divergence-free part $\bu^0_h$} fulfills
a \emph{pressure-robust} discretization of the
\emph{incompressible
Stokes} problem in the form:
for all $\bv^0_h \in \bV^0_h$ it holds
$$
  a_1(\bu^0_h, \bv^0_h) = F(\Pi \vecb{v}^0_h) + G(\varrho_h, \Pi \vecb{v}^0_h)
    = (\mathbb{P}(\bfe + \rho_h \bg), \Pi \bv^0_h) ).
$$
\end{theorem}
%%%%%%%%%%%%%
\begin{remark}\label{rem:asymptotic_convergence}
Theorem \ref{thm:grad:robust:structure} is the discrete
equivalent to the continuous relation
\eqref{eq:u0:helm}.
We emphasize the appearance of the \emph{continuous}
Helmholtz--Hodge projector $\mathbb{P}(\bfe + \rho_h \bg)$.
Actually, it is again Theorem \ref{thm:grad:robust:structure}
that makes the scheme \emph{asymptotic-preserving} in the
low Mach number limit, where the non-divergence-free
part $\bu^\perp_h$ of the discrete velocity solution
$\bu_h$ should vanish in the limit.
\end{remark}

%Instead, the proposed discretization delivers
%a kind of a well-balanced scheme --- well-known from
%hyperbolic conservation laws --- which only exploits
%an $\bL^2$ orthogonality of discretely divergence-free
%vector fields and arbitrary gradient fields.

Consider the compressible Stokes problem \eqref{eqn:compressible_weak} with the right-hand sides
\begin{align*}
  \vecb{f} := \nabla q \quad \text{and} \quad \vecb{g} = 0
\end{align*}
for some \(q \in H^1(\Omega)\). For this setting, one can observe that
the solution \((\vecb{u},p) = (\vecb{0},q+C)\) of the incompressible Stokes problem also solves the compressible problem if there is enough mass
in the system. Indeed, if it exists a (global) constant \(C\), such that \(\varrho := q/c+C\) satisfies the mass constraint \(\int_\Omega \varrho \, dx = M\) and is non-negative \(\varrho \geq 0\), then the solution of the incompressible
Stokes problem also is a solution of the compressible problem.

Vice versa, assume that \(\varrho \geq 0\) satisfies the mass constraint and \(\nabla(\varphi(\varrho)) = \nabla q\). Then, it is clear that
\((\vecb{u},p,\varrho) = (\vecb{0},\varphi(\varrho),\varrho)\) solves the
compressible Stokes problem and \((\vecb{u},p) = (\vecb{0},\varphi(\varrho))\)
solves the incompressible Stokes problem. This proves the following lemma.

\begin{lemma}\label{lem:hydrostatic_solutions_compressible}
  The compressible Stokes problem with right-hand sides
  \begin{align*}
    \vecb{f} := \nabla q \quad \text{and} \quad \vecb{g} = 0
  \end{align*}
  has a hydrostatic solution \(\vecb{u} = \vecb{0}\), if and only if it exists a (global) constant \(C\), such that \(\varrho := q/c+C\) satisfies the mass constraint \(\int_\Omega \varrho \, dx = M\) and is non-negative, i.e.\ \(\varrho \geq 0\). The pair \((\vecb{u},p) := (\vecb{0},q)\) also solves
  the incompressible Stokes problem.
\end{lemma}

\begin{definition}[Well-balanced property]
  A discretization of the compressible Stokes problem is called \emph{well-balanced} if it computes hydrostatic solutions \(\vecb{u} = \vecb{0}\) correctly
  if the right-hand side is balanced by the gradient of some admissible
  pressure-density pair.
\end{definition}

%\begin{remark}
%  The notion \emph{well-balanced} stems from the shallow water community that is interested in finding solutions \((\vecb{u},h)\) of the problem
%    \begin{align*}
%       \mathrm{div} (h \vecb{u}) & = 0,\\
%  	   \mathrm{div}(h \vecb{u} \otimes \vecb{u}) + gh \nabla h & = - gh \nabla b.
%    \end{align*}  
%  Here, \(h\) denotes the water column depth, \(\vecb{u}\) is some depth-averaged horizontal velocity and \(b\) is some bottom profile. Structurally, especially the continuity equation, this problem is
%  similar to the compressible Stokes problem without the friction force.
%  Within this community, discretizations are called well-balanced if they compute exact velocities
%  in certain steady scenarios, e.g.\ in the lake-at-rest scenario where the scheme
%  has to balance the gradient of the bottom profile \(b\) and the gradient of \(h\),
%  i.e. \(\vecb{u} = 0\) and \(\nabla (h + b ) = 0\), see e.g.\ \cite{Berthon:2012:EWH:2240300.2240375} for some background.
%  
%  In the context of the incompressible Stokes problem, such a discretization is
%  called pressure-robust, if it computes \(\vecb{u} = \vecb{0}\) whenever the
%  right-hand side is a gradient, i.e.\ when the discrete velocity is independent of the exact pressure \cite{cmame:linke:merdon:2016,JLMNR:sirev}.
%  \todo{elaborate on that}
%\end{remark}

%%%%%%%%%%%%%%%%%%%%%%%%%%%%%%%%%%%%%%%%%%%%%%%%%%%%%%%%%%%%%%%%%%%%%%
\section{Existence of discrete solutions}\label{sec:existence}

The discussion in Subsection \ref{subsection:pseudo:time:stepping}
and in
Section \ref{sec:gradient_robustness}
motivates the following
pseudo-time stepping algorithm
and the choice of its initial value.
Subsection \ref{subsection:fixed:point} proves that this algorithm
has a fixed point, which is a discrete solution of \eqref{eq:gradient:robust:scheme}.

\subsection{An iterative algorithm with well-balanced initial solution}\label{sec:algorithm}
The previous discussion motivates to choose the initial solution by a solve of
the incompressible Stokes equations and a rescaling of its pressure. In case of a well-balanced situation as in Lemma~\ref{lem:hydrostatic_solutions_compressible}, this then already gives a discrete solution of the compressible system. Otherwise, one enters a suitable fixed point iteration.

\medskip
\textbf{Input.}
\begin{itemize}
  \item some triangulation \(\mathcal{T}\),
  \item stepsize \(\tau > 0\).
\end{itemize}

\medskip 
\textbf{Initial Step.}
\begin{itemize}
\item 
  Set \(\varrho_{-1} \equiv M/\lvert \Omega \rvert\).
\item Solve the incompressible Stokes system, i.e., find \(\vecb{u}_0 \in \vecb{V}_h\) and \(p_0 \in Q_h\) such that
  \begin{align*}
  a_1(\vecb{u}_0,\vecb{v}_h) + b(p_0,\vecb{v}_h) & = F(\Pi \vecb{v}_h) + G(\varrho_{-1}, \Pi \vecb{v}_h) && \text{for all } \vecb{v} \in \vecb{V}_h,\\
  b(q_h,\vecb{u}_0) & = 0 && \text{for all } q_h \in Q_h.
\end{align*}
\item Set \(\varrho_0 := p_0/c + C\), where \(C \in \mathbb{R}\) is chosen such that
\(\varrho_0\) satisfies \(\int_\Omega \varrho_0 \, dx = M\) and
\(\varrho_0 \geq 0\).
%(This requires some zero search algorithm if \(\varphi\) is nonlinear, here we use a simple bisection strategy).
If this is not possible, start with \(\varrho_0 = \varrho_{-1}\) and
 \(\vecb{u}_h = 0\).
\end{itemize}

\medskip
\textbf{Loop (start with \(n=1\)).}
\begin{itemize}
  \item Update matrix \(D\) according to \eqref{eqn:definition_matrixD} (with \(\bu_h = \bu_h^{n-1}\)) and find \(\varrho_h^n \in Q_h\) such that
    \begin{align} \label{eq:update:disc:rho}
    (M + \tau D) \varrho_h^{n} = M \varrho_h^{n-1}.
  \end{align}
  \item Update the pressure according to the equation of state, i.e.
  \begin{align} \label{eq:update:disc:p}
    p_h^n := \varphi(\varrho_h^n) = c \varrho_h^n.
  \end{align}
  \item Find \(\vecb{u}_h^n \in \vecb{V}_h\) that
  satisfies the momentum equation
    \begin{align} \label{eq:update:disc:u}
    \hspace{1cm}  a_1(\vecb{u}_h^n,\vecb{v}_h) + a_{2}(\Pi \vecb{u}_h^n,\Pi \vecb{v}_h) & = F(\Pi \vecb{v}_h) + G(\varrho_h^{n}, \Pi \vecb{v}_h) - b(p_h^{n},\vecb{v}_h) && \text{for all } \vecb{v}_h \in \vecb{V}_h.
  \end{align}
  \item Compute residuals of the stationary momentum equation and the continuity equation, i.e.
  \begin{align*}
    \hspace{1cm} \text{res} := \| a_1(\vecb{u}_h^n,\bullet) + a_{2}(\Pi \vecb{u}_h^n,\Pi \bullet) - F(\Pi \bullet) + G(\varrho_h^{n}, \Pi \bullet) - b(p_h^{n},\bullet) \|_{l^2} + \lvert \mathrm{div}_\text{upw}(\varrho_h^n \vecb{u}_n) \rvert
  \end{align*}
  \item Stop if \(\text{res} < \text{tol}\), otherwise increase \(n\) by one and restart loop.
\end{itemize}

\begin{remark}
Note, that one only can prove that there exists some discrete solution (see Subsection~\ref{subsection:fixed:point}), but it is not guaranteed that the algorithm converges. In our numerical benchmarks, we could
enforce convergence by choosing small enough time steps \(\tau\).
\end{remark}

\subsection{Existence of a fixed point}\label{subsection:fixed:point}
Note, that there is no uniqueness result for the continuous compressible Stokes system, but one can show existence of a (discrete) solution for the (discretized) compressible Stokes problem. To do so we mainly follow the argumentation in \cite{gal-09-conv}.
There the existence of a weak solution with \(\varrho \in L^{2}\) and \(p = \varphi(\varrho) := c \varrho\) is proven. The main argument concerns the proof of the a priori stability
  estimate
  \begin{align*}
    \min \lbrace 2\mu + \lambda, \mu \rbrace \| \nabla \vecb{u}_h \|_{L^2} + \| p_h \|_{L^2(\Omega)} + \| \varrho_h \|_{L^{2}(\Omega)} & \lesssim 1
  \end{align*}  
  which is needed in the convergence proof via some Brouwer fixed point argument.
  The crucial point is the vanishing term
  \begin{align}\label{eqn:divpressureintegral_vanish}
    \int_\Omega p \mathrm{div} (\vecb{u})
    = c\int_\Omega \varrho \mathrm{div} (\vecb{u})
    = - c\int_\Omega \nabla(\mathrm{log}\varrho) \cdot (\varrho \vecb{u})
    = c\int_\Omega \mathrm{log}\varrho \, \mathrm{div}(\varrho \vecb{u})
    = 0
  \end{align}
  for \(\varrho \in C^1(\overline{\Omega})\), which can also be generalized to \(\varrho \in L^{2}(\Omega)\), see \cite[Appendix A]{MR2600538} for details.
  
  A similar stability estimate holds for the discrete scheme which requires the following Lemma.
  \begin{lemma}\label{lem:convexity_divergence_lemma}
  For any convex and twice continuously differentiable function \(\phi : [0, \infty) \rightarrow \mathbb{R}^+\), it holds
  \begin{multline*}
     \int_\Omega \phi^\prime(\varrho_h) \mathrm{div}_\text{upw}(\varrho_h \vecb{u}_h) \, dx
     - \int_\Omega (\varrho_h \phi^\prime(\varrho_h) - \phi(\varrho_h)) \mathrm{div}(\vecb{u}_h) \, dx\\
     = \frac{1}{2} \sum_{F_{KL} \in \mathcal{F}(\Omega)} \phi^{\prime\prime}(\varrho_{KL})\frac{\lvert u_{K,F_{KL}} \rvert }{\lvert F_{KL} \rvert}  \| [[ \varrho_h ]] \|^2_{L^2(F_{KL})} \geq 0,
  \end{multline*}  
  where the quantities
  \(\varrho_{KL} \in (\varrho_h|_K,\varrho_h|_L)\)
  denote intermediate values on every face $F_{KL} \in \mathcal{F}(\Omega)$
  according to remainders in corresponding Taylor expansions.
   \end{lemma}
   \begin{proof}
   By convexity of \(\phi\) and Taylor expansion, it holds
   \begin{align}\label{eqn:convexity_property}
     \phi^\prime(x)(x-y) - \phi(x) + \phi(y) = \frac{1}{2} \phi^{\prime\prime}(s) (x-y)^2 \geq 0
     \quad \text{for some } s \in (x,y).
   \end{align}   
   The integrals in the assertion can be rewritten into
   \begin{align*}
     \int_\Omega \phi^\prime(\varrho_h) \mathrm{div}_\text{upw}(\varrho_h \vecb{u}_h) \, dx
     & = \sum_{T \in \mathcal{T}} \sum_{F \in \mathcal{F}(T)} \phi^\prime(\varrho_h|_T) \varrho_F^\text{upw} u_{T,F}\\
     \int_\Omega (\varrho_h \phi^\prime(\varrho_h) - \phi(\varrho_h)) \mathrm{div}(\vecb{u}_h) \, dx
     & = \sum_{T \in \mathcal{T}} \sum_{F \in \mathcal{F}(T)} (\varrho_h|_T \phi^\prime(\varrho_h|_T) -\phi( \varrho_h|_T)) u_{T,F}.
   \end{align*}
   Hence, their difference reads
   \begin{multline*}
     \int_\Omega \phi^\prime(\varrho_h) \mathrm{div}_\text{upw}(\varrho_h \vecb{u}_h) \, dx - \int_\Omega (\varrho_h \phi^\prime(\varrho_h) - \phi(\varrho_h)) \mathrm{div}(\vecb{u}_h) \, dx\\
     \begin{aligned}
     &=\sum_{T \in \mathcal{T}} \sum_{F \in \mathcal{F}(T)} \left(
     \phi^\prime(\varrho_h|_T) \varrho_F^\text{upw} - \varrho_h|_T \phi^\prime(\varrho_h|_T) + \phi(\varrho_h|_T)
     \right) u_{T,F}\\
     &=\sum_{F_{KL} \in \mathcal{F}(\Omega)} 
     \left(
     \phi^\prime(\varrho_h|_K)(\varrho_{F_{KL}}^\text{upw} - \varrho_h|_K) + \phi^\prime(\varrho_h|_L)(\varrho_h|_L - \varrho_{F_{KL}}^\text{upw}) + \phi(\varrho_h|_K) - \phi(\varrho_h|_L)
     \right) u_{K,F_{KL}}\\
     & =: \sum_{F_{KL} \in \mathcal{F}(\Omega)} u_{K,F_{KL}} \theta_{KL}
     \end{aligned}
   \end{multline*}
   where the last sum collects the flux jumps \(\theta_{KL}\) over all
   interior faces \(F_{KL} \in \mathcal{F}(\Omega)\) (on boundary faces it holds \(u_{K,F_{KL}} = 0\)). It remains to show that each summand is non-negative.
   The first case assumes \(u_{K,F_{KL}} > 0\) and hence \(\varrho_{F_{KL}}^\text{upw} = \varrho_h|_K\). Then, one obtains for the jump term
   \begin{align*}
      \theta_{KL} & = \phi^\prime(\varrho_h|_L)(\varrho_h|_L - \varrho_h|_K) + \phi(\varrho_h|_K) - \phi(\varrho_h|_L) 
      = \phi^{\prime\prime}(\varrho_{KL}) (\varrho_h|_K - \varrho_h|_L)^2 \geq 0
   \end{align*}
   due to \eqref{eqn:convexity_property} where \(s\) is renamed
   to \(\varrho_{KL}\).
   In the other case \(u_{K,F_{KL}} < 0\) it holds
   \(\varrho_{F_{KL}}^\text{upw} = \varrho_h|_L\) and hence  
   \begin{align*}
      \theta_{KL} & = \phi^\prime(\varrho_h|_K)(\varrho_h|_L - \varrho_h|_K) + \phi(\varrho_h|_K) - \phi(\varrho_h|_L) 
      = - \phi^{\prime\prime}(\varrho_{KL}) (\varrho_h|_K - \varrho_h|_L)^2 \leq 0
   \end{align*}
   again by \eqref{eqn:convexity_property} (multiplied by \(-1\)).
   Hence
   \begin{multline*}
   \int_\Omega \phi^\prime(\varrho_h) \mathrm{div}_\text{upw}(\varrho_h \vecb{v}_h) \, dx - \int_\Omega (\varrho_h \phi^\prime(\varrho_h) - \phi(\varrho_h)) \mathrm{div}(\vecb{v}_h) \, dx\\
     \begin{aligned}
   & = \sum_{F_{KL} \in \mathcal{F}(\Omega)} \lvert u_{K,F_{KL}} \rvert \phi^{\prime\prime}(\varrho_{KL}) (\varrho_h|_K - \varrho_h|_L)^2\\
   & = \frac{1}{2} \sum_{F_{KL} \in \mathcal{F}(\Omega)} \phi^{\prime\prime}(\varrho_{KL})\frac{\lvert u_{K,F_{KL}} \rvert }{\lvert F_{KL} \rvert}  \| [[ \varrho_h ]] \|^2_{L^2(F_{KL})}
    \geq 0.
   \end{aligned}
   \end{multline*}
   This concludes the proof.
   \end{proof}
  
  %%%%%%%%%%%%%%%%%%%%%%%%%%%%%%%%%%%%%%%%%%%%%%%%%%%%%%%%%%%%%%%%%%%%%%
  \begin{lemma}[Discrete stability estimate]\label{lem:discrete-stability}
  For any solution \((\vecb{u}_h,p_h,\varrho_h)\) of the discrete scheme, it holds
  \begin{align}
    & \label{eq:disc:stab:est:i} & 
    \min \lbrace 2\mu + \lambda, \mu \rbrace \| \nabla \vecb{u}_h \|_{L^2} 
    & \lesssim \| \vecb{f} \|_{L^2} + \| \varrho_h \|_{L^{2}(\Omega)} \| \vecb{g} \|_{L^{\infty}},\\
    & \label{eq:disc:stab:est:ib} & 
    \sum_{F_{KL} \in \mathcal{F}(\Omega)} \varrho_{KL}^{-1}\frac{\lvert u_{K,F_{KL}} \rvert}{\lvert F_{KL} \rvert}   \| [[ \varrho_h ]] \|^2_{L^2(F_{KL})}
    & \lesssim c^{-1} \min \lbrace 2\mu + \lambda, \mu \rbrace^{-1} \left(\| \vecb{f} \|_{L^2} + \| \varrho_h \|_{L^{2}(\Omega)} \| \vecb{g} \|_{L^{\infty}}\right),\\
    & & \text{If } \vecb{g} \equiv \vecb{0} \ \Longrightarrow \ \label{eq:disc:stab:est:ii} \| p_h \|_{L^2(\Omega)} & = c\| \varrho_h \|_{L^{2}(\Omega)} \lesssim \| \vecb{f} \|_{L^2} + c.
  \end{align}
  \end{lemma}
  The hidden constants in \(\lesssim\)
  depend neither
  on the mesh width \(h\), nor
  the viscosity parameters \(\mu\) and \(\lambda\) nor on \(c\).
 % Also note, that the same estimates can be proven for the exact solution in a similar fashion.
  \begin{proof}  
  
  Testing the momentum equation with \(\vecb{u}_h\) yields
  \begin{align}\label{eqn:momentum_test}
    2\mu \| \epsilon(\vecb{u}_h) \|_{L^2}^2
    + \lambda \| \mathrm{div}_h \vecb{u}_h \|_{L^2}^2
    - \int_\Omega p_h \mathrm{div} \vecb{u}_h \, dx
    = \int \vecb{f} \cdot \Pi \vecb{u}_h \, dx + \int \varrho_h \vecb{g} \cdot \Pi \vecb{u}_h \, dx.
  \end{align}
  The approximation properties of \(\Pi\) yield
  \begin{align*}
    \int \vecb{f} \cdot \Pi \vecb{u}_h \, dx
    & \leq (C_F + h C_\Pi) \| \vecb{f} \|_{L^2} \| \nabla \vecb{u}_h \|_{L^2}.
  \end{align*}
  The integral with \(\vecb{g}\) is estimated similarly by
  \begin{align*}
    \int \varrho_h \vecb{g} \cdot \Pi \vecb{u}_h \, dx
    \leq (C_F + h C_\Pi) \| \varrho_h \|_{L^{2}(\Omega)} \| \vecb{g} \|_{L^{\infty}} \| \nabla \vecb{u}_h \|_{L^2}
  \end{align*}  
  and it remains to handle the integral on the left-hand side.
 Lemma~\ref{lem:convexity_divergence_lemma} shows, for \(\phi(s) := s \log(s)\) and with
  \(\mathrm{div}_\text{upw}(\varrho_h \vecb{u}_h)=0\) due to
  \eqref{eq:gradient:robust:scheme}, that
  \begin{align}
  \int_\Omega p_h \mathrm{div}(\vecb{u}_h) \, dx
  & = c \int_\Omega \varrho_h \mathrm{div}(\vecb{u}_h) \, dx\nonumber\\
  & = c \int_\Omega (1+\log(\varrho_h)) \mathrm{div}_\text{upw}(\varrho_h \vecb{u}_h) \, dx
  - c \sum_{F_{KL} \in \mathcal{F}(\Omega)} \frac{\phi^{\prime\prime}(\varrho_{KL})}{2 \lvert F_{KL} \rvert} \lvert u_{K,F_{KL}} \rvert  \| [[ \varrho_h ]] \|^2_{L^2(F_{KL})}\nonumber\\
  & = - c \sum_{F_{KL} \in \mathcal{F}(\Omega)} \frac{\varrho_{KL}^{-1}}{2 \lvert F_{KL} \rvert} \lvert u_{K,F_{KL}} \rvert  \| [[ \varrho_h ]] \|^2_{L^2(F_{KL})} \leq 0.\label{eqn:jumpterm}
  \end{align}

Assume that \(\lambda \geq 0\). Then, it holds
\begin{align*}
  \lambda \| \mathrm{div}_h \vecb{u} \|^2_{L^2}
  \geq 0
\end{align*}  
and hence
\begin{align*}
    \mu \| \nabla \vecb{u}_h \|^2_{L^2} \leq 2\mu \| \epsilon(\vecb{u}_h) \|_{L^2}^2
    \lesssim \left( \| \vecb{f} \|_{L^2} + \| \varrho_h \|_{L^{2}(\Omega)} \| \vecb{g} \|_{L^{\infty}} \right) \| \nabla \vecb{u}_h \|_{L^2}.
  \end{align*}
  Division by \(\| \nabla \vecb{u}_h \|_{L^2}\) concludes the proof in this case.

In the case \(0 > \lambda > -2\mu\), elementary vector calculus identities yield
\begin{align*}
  \min\lbrace\mu, 2\mu + \lambda\rbrace \| \nabla \vecb{u}_h \|^2_{L^2}
  \leq \mu \| \mathrm{rot} \vecb{u} \|^2_{L^2} + (2\mu + \lambda) \| \mathrm{div} \vecb{u} \|^2_{L^2}
  = 2 \mu \| \epsilon(\vecb{u}_h) \|^2_{L^2}
  + \lambda \| \mathrm{div} \vecb{u} \|^2_{L^2}.
\end{align*}
Moreover, it holds
\begin{align*}
  0 \geq \lambda \| \mathrm{div}_h \vecb{u} \|^2_{L^2} \geq \lambda \| \mathrm{div} \vecb{u} \|^2_{L^2}
\end{align*}
and hence
\begin{align*}
    \min\lbrace\mu, 2\mu + \lambda\rbrace \| \nabla \vecb{u}_h \|^2_{L^2} \leq 2\mu \| \epsilon(\vecb{u}_h) \|_{L^2}^2 + \lambda \| \mathrm{div}_h \vecb{u} \|^2_{L^2}
    \lesssim \left( \| \vecb{f} \|_{L^2} + \| \varrho_h \|_{L^{2}(\Omega)} \| \vecb{g} \|_{L^\infty} \right) \| \nabla \vecb{u}_h \|_{L^2}.
  \end{align*}
  Division by \(\| \nabla \vecb{u}_h \|_{L^2}\) concludes the proof of
  \eqref{eq:disc:stab:est:i}. The proof of \eqref{eq:disc:stab:est:ib}
  follows from a combination of \eqref{eqn:momentum_test} and \eqref{eqn:jumpterm} together with the already proven estimate \eqref{eq:disc:stab:est:i}, i.e.
  \begin{align*}
     \sum_{F_{KL} \in \mathcal{F}(\Omega)} \varrho_{KL}^{-1}\frac{\lvert u_{K,F_{KL}} \rvert}{\lvert F_{KL} \rvert}   \| [[ \varrho_h ]] \|^2_{L^2(F_{KL})}
    & \lesssim c^{-1} \left( \| \vecb{f} \|_{L^2} + \| \varrho_h \|_{L^{2}(\Omega)} \| \vecb{g} \|_{L^{\infty}} \right) \| \nabla \vecb{u}_h \|_{L^2}\\
    & \lesssim c^{-1} \min \lbrace 2\mu + \lambda, \mu \rbrace^{-1} \left( \| \vecb{f} \|_{L^2} + \| \varrho_h \|_{L^{2}(\Omega)} \| \vecb{g} \|_{L^{\infty}} \right)^2.
  \end{align*}
  
  For the proof of \eqref{eq:disc:stab:est:ii}, consider a test function \(\vecb{v}_h\) with \(\mathrm{div}_h(\vecb{v}_h) = p_h - \overline{p}_h\), where \( \overline{p}_h := \lvert \Omega\rvert^{-1} \int_\Omega p_h \, dx\), 
  and \(\| \nabla \vecb{v}_h \|_{L^2} \lesssim \| p_h - \overline{p}_h \|_{L^2}\), using discrete inf-sup stability. Inserting this test function in the momentum equation and using also bounds from \eqref{eq:disc:stab:est:i}, it follows the estimate
  \begin{align}
    \| p_h - \overline{p}_h \|_{L^2}^2 & = \lambda \int_\Omega \mathrm{div}_h(\vecb{u}_h) \mathrm{div}_h(\vecb{v}_h) \,dx
    + 2\mu \int_\Omega \epsilon(\vecb{u}_h) : \epsilon(\vecb{v}_h) \,dx
    - \int \vecb{f} \cdot \Pi \vecb{v}_h \, dx\nonumber\\
    & \lesssim \| \vecb{f} \|_{L^2} \| p_h - \overline{p}_h \|_{L^2} \label{eqn:estimate:p-pmean}
  \end{align}
  where the constant \(C\) also depends on \(\| \vecb{f} \|_{L^2}\) (note that we assumed here that \(\vecb{g} \equiv 0\)).
  For the following estimate we also need a bound on \(\| \overline{p}_h \|_{L^2}\) that can be obtained due to
  \begin{align} \label{eqn:estimate:pmean}
    \overline{p}_h = \frac{1}{\lvert \Omega \rvert}\int_\Omega p_h \, dx = \frac{c}{\lvert \Omega \rvert} \int_\Omega \varrho_h \, dx = \frac{cM}{\lvert \Omega \rvert}.
  \end{align}
  A Pythagoras argument and the combination of \eqref{eqn:estimate:p-pmean} and \eqref{eqn:estimate:pmean} results in
  \begin{align*}
    \| p_h \|^2_{L^2} = \| p_h - \overline{p}_h \|^2_{L^2} + \| \overline{p}_h \|^2_{L^2}
    \lesssim \| \vecb{f} \|_{L^2}^2 + \frac{(cM)^2}{\lvert \Omega \rvert}  \lesssim \| \vecb{f} \|_{L^2}^2 + c^2.
  \end{align*}
  This concludes
   the proof of \eqref{eq:disc:stab:est:ii}.
    \end{proof}

%%%%%%%%%%%%%%%%%%%%%%%%%%%%%%%%%%%%%%%%%%%%%%%%
\begin{lemma}[Existence of a discrete solution]
On every (fixed) shape-regular mesh in the sense of
Section \ref{sec:discrete_scheme}, the discrete nonlinear equation system
\eqref{eq:gradient:robust:scheme}
has at least one solution.
\end{lemma}
\begin{proof}
The existence of a discrete weak solution \((\vecb{u}_h,p_h,\varrho_h) \in \vecb{V}_h \times Q_h \times Q_h\) is proved by the
Brouwer fixed-point theorem.
We start the algorithm presented in Subsection \ref{sec:algorithm}
with \(\varrho_{-1} \equiv M/\lvert \Omega \rvert\)
and $\vecb{u}_{-1} = \vecb{0}$, $p_{-1}=\varphi(\varrho_{-1})$.
Obviously, the discrete
start value \((\vecb{u}_{-1},p_{-1},\varrho_{-1}) \in \vecb{V}_h \times Q_h \times Q_h\) lies in a
finite-dimensional product space of convex sets
with finite diameter that is itself convex.
Now a mapping
\begin{align*}
f: \vecb{V}_h \times Q_h \times Q_h &\to \vecb{V}_h \times Q_h \times Q_h,\\
(\vecb{u}_n,p_n,\varrho_n) &\mapsto f(\vecb{u}_n,p_n,\varrho_n) :=
(\vecb{u}_{n+1},p_{n+1},\varrho_{n+1})
\end{align*}
is constructed by composition
of \eqref{eq:update:disc:rho}, \eqref{eq:update:disc:p}
and \eqref{eq:update:disc:u} where an (arbitrary)
stepsize $\tau > 0$ is fixed.
Then, the mapping defined by \eqref{eq:update:disc:rho}
is linear and continuous, since $(M + \tau D)$ is an invertible matrix.
The mapping defined by \eqref{eq:update:disc:p} is again continuous.
Finally, the mapping defined by \eqref{eq:update:disc:u} is
linear and continuous, since the assumptions on the viscosities
$\mu$ and $\lambda$ assure that the discrete bilinear form
$a_1(\vecb{u}_h,\vecb{v}_h) + a_{2}(\Pi \vecb{u}_h, \Pi \vecb{v}_h)$
is coercive. Therefore, the composed
mapping $f$ constructed in the algorithm
in Subsection \ref{sec:algorithm} is continuous.

Due to discrete mass conservation in \eqref{eq:update:disc:rho} (see
Lemma~\ref{lem:mass_conservation}), it follows
$$\|\rho_{n+1}\|_{L^1}=\|\rho_{n}\|_{L^1}=\ldots=\|\rho_{-1}\|_{L^1}=M.$$
Due to this and the equivalence of all norms in finite dimensions,
it holds $\|p_n \|_{L^2} = c\| \rho_n  \|_{L^{2}} \leq C(h)$ for all \(n \geq -1\).
Hence, all $(p_n,\varrho_n)$ for \(n \geq -1\) lie in the same
convex set.
Finally, similar to the proof of \eqref{eq:disc:stab:est:i}, a discrete
bound can be proved for $\| \nabla \vecb{u}_{n+1}\|_{L^2}$.
The only difference is that one cannot assume that $\vecb{u}_{n+1}$ fulfills the discrete mass conservation in
\eqref{eq:gradient:robust:scheme} with $\mathrm{div}_\text{upw}(\varrho_{n+1} \vecb{u}_{n+1}) = 0$.
Therefore, the term $(p_{n+1}, \mathrm{div}(\vecb{u}_{n+1}))$ has to be estimated
in a different way.
However, since the grid is fixed, the argument above yields
$$
  \left | (p_{n+1}, \mathrm{div}(\vecb{u}_{n+1})) \right |
    \leq \| p_{n+1} \|_{L^2} \| \nabla \vecb{u}_{n+1} \|_{L^2}
    \leq C(h) \| \nabla \vecb{u}_{n+1} \|_{L^2},
$$
and one derives a similar estimate like in \eqref{eq:disc:stab:est:i}.
Therefore, $f$ is a continuous function that maps a convex set into itself.
According to the Brouwer fixed point theorem, this mapping has a fixed-point that is a solution of \eqref{eq:gradient:robust:scheme}. 
\end{proof}

\section{Convergence of the scheme}\label{sec:convergence}
This section proves convergence of the discrete solutions
to a weak solution of \eqref{eqn:compressible_weak}.

\begin{theorem}\label{thm:convergence}
We assume that it holds $\bg = \vecb{0}$.
  Consider a sequence of shape-regular triangulations \((\mathcal{T}_k)_{k \in \mathbb{N}})\) with mesh width \(h_k \rightarrow 0\).
  Let \((\vecb{u}_k,p_k,\varrho_k)\) denote the
  corresponding discrete solutions of \eqref{eq:gradient:robust:scheme} on the meshes \(\mathcal{T}_k\). Then,
 up extraction of a subsequence, it holds
 \begin{align*}
   && (i)
   & \text{ the sequence } (\vecb{u}_k)_{k \in \mathbb{N}}
   \text{ converges weakly/strongly in } \vecb{H}^1_0(\Omega) / \vecb{L}^2(\Omega) \text{ to a limit } \vecb{u} \in \vecb{H}^1_0(\Omega),\\
   && (ii)
   & \text{ the sequence } (p_k)_{k \in \mathbb{N}} = (c \varrho_k)_{k \in \mathbb{N}}
   \text{ converges weakly in } L^2(\Omega) \text{ to a limit } p = c \varrho \in L^2(\Omega), \\   
   && (iii)
   & \text{ the limit } (\vecb{u},p,\varrho) \text{ is a solution of  \eqref{eqn:compressible_weak}.}  
 \end{align*}
\end{theorem}
\begin{proof}
  The weak convergence and the existence of the limit \((\vecb{u},p,\varrho) \in \vecb{H}^1_0(\Omega) \times L^2(\Omega) \times L^{2}(\Omega)\) follows from Lemma~\ref{lem:discrete-stability} and
  standard arguments from linear functional analysis.
%  The strong convergence in the weaker spaces
%  follows from the Rellich–Kondrachov theorem and similar embeddings
%  theorem for \(L^q(\Omega) \subset \subset L^p(\Omega)\) for \(q < p\).
Hence, it remains to prove (iii). 
  
  \medskip
  \textbf{Step 1. $(\vecb{u},p)$ satisfy the momentum equation, i.e.,\
  the first equation of \eqref{eqn:compressible_weak}.}
  
  Consider an arbitrary test function \(\vecb{v} \in \vecb{C}^{\infty}_0\), which is dense in \(\vecb{H}^1_0(\Omega\)), and
  define $\vecb{v}_k \in \vecb{V}_k$
    on \(\mathcal{T}_k\) as its best approximation in the
    $\vecb{H}^1_0$ norm, i.e., it holds for all
    $\vecb{w}_h \in \vecb{V}_h$:
    $(\nabla \vecb{v}_h, \nabla \vecb{w}_h) = (\nabla \vecb{v}, \nabla \vecb{w}_h)$. From standard arguments follows
    the strong convergence
  \begin{align}\label{eqn:strong_convergence_vk}
    \vecb{v}_h \to \vecb{v} \text{ in } \vecb{H}^1_0(\Omega).
  \end{align}
  This strong convergence and the weak convergence of \((\vecb{u}_k)\)
  to \(\vecb{u}\) in \(\vecb{H}^1_0(\Omega)\) allows to conclude
  \begin{align*}
    \int_\Omega \epsilon(\vecb{u}_k) : \epsilon(\vecb{v}_k) \, dx
    \ \rightarrow \
    \int_\Omega \epsilon(\vecb{u}) : \epsilon(\vecb{v}) \, dx.
  \end{align*} 
  Similarly, \eqref{eqn:strong_convergence_vk} and the weak convergence
  of \((p_k)\) to \(p\) in \(L^2(\Omega)\) yield
  \begin{align*}
    \int_\Omega p_k \mathrm{div}(\vecb{v}_k) \, dx
    \ \rightarrow \
    \int_\Omega p \mathrm{div}(\vecb{v}) \, dx.
  \end{align*} 
  Since also \(\| \mathrm{div}_h (\vecb{v}) - \mathrm{div}(\vecb{v})\|_{L^2(\Omega)}
  \rightarrow 0\), it follows
  \begin{align*}
    \int_\Omega \mathrm{div}_h(\vecb{u}_k) \mathrm{div}_h(\vecb{v}_k) \, dx
    &= \int_\Omega \mathrm{div}(\vecb{u}_k) \mathrm{div}(\vecb{v}) \, dx + \int_\Omega \mathrm{div}(\vecb{u}_k) \left(\mathrm{div}_h(\vecb{v}_k) - \mathrm{div}(\vecb{v})\right)\, dx\\
    &\rightarrow \
    \int_\Omega \mathrm{div}(\vecb{u}) \mathrm{div}(\vecb{v}) \, dx.
  \end{align*}
  It remains to show convergence of the right-hand side integrals, which follows
  again by \eqref{eqn:strong_convergence_vk} and the weak convergence
  of \((\varrho_k)\) to \(\varrho\) in \(L^2(\Omega)\), i.e.
  \begin{align*}
    \int_\Omega \vecb{f} \cdot \vecb{v}_k \, dx
    \ \rightarrow \
    \int_\Omega \vecb{f} \cdot \vecb{v} \, dx.
  \end{align*} 
  The combination of all convergence results concludes the proof of Step 1.

  \medskip
  \textbf{Step 2. $(\vecb{u},\varrho)$ satisfy the continuity equation, i.e.\
  the second equation of \eqref{eqn:compressible_weak}.}

We define on every element $T  \in \mathcal{T}$
\begin{equation}
  \vecb{q}_h|_T := I_h^{\mathrm{RT}_0} (\varrho_{\text{upw}} \vecb{u}_h|_T),
\end{equation}
which is $\vecb{H}(\mathrm{div}, \Omega)$-conforming and divergence-free.
For an arbitrary scalar $P_1$ function $\psi_h$, it holds
\begin{equation*}
  0  = (\psi_h, \mathrm{div} \, \vecb{q}_h) \\
     = - (\vecb{q}_h, \nabla \psi_h) \\
     = -( \vecb{q}_h - \varrho_h \vecb{u}_h, \nabla \psi_h)
    - (\varrho_h \vecb{u}_h, \nabla \psi_h).
\end{equation*}
We estimate now the term
\begin{equation*}
\begin{split}
 \left | ( \vecb{q}_h - \varrho_h \vecb{u}_h, \nabla \psi_h) \right | &
 = \left | \sum_T \nabla \psi_h|_T \cdot 
  \int_T (\vecb{q}_h - \varrho_h \vecb{u}_h) \, dx \right |.
\end{split}
\end{equation*}
For the term under the integral we get
\begin{equation*}
\begin{split}
   \int_T (\vecb{q}_h - \varrho_h \vecb{u}_h) \, dx
     & =  \int_T (\vecb{q}_h - \varrho_h I_h^{\mathrm{RT}_0} \vecb{u}_h) \, dx + \int_T \varrho_h (I_h^{\mathrm{RT}_0} \vecb{u}_h - \vecb{u}_h) \, dx.
\end{split}
\end{equation*}
Thus one obtains by the triangle inequality
\begin{equation}
\begin{split}
  \left | ( \vecb{q}_h - \varrho_h \vecb{u}_h, \nabla \psi_h) \right |
    & \leq  \sum_T \left | \nabla \psi_h|_T \cdot \left ( \int_T (\vecb{q}_h - \varrho_h I_h^{\mathrm{RT}_0} \vecb{u}_h) \, dx + \int_T \varrho_h (I_h^{\mathrm{RT}_0} \vecb{u}_h - \vecb{u}_h) \, dx  \right ) \right | \\
     & \leq C \sum_T \left ( \| \vecb{q}_h - \varrho_h I_h^{\mathrm{RT}_0} \vecb{u}_h \|_{L^1(T)} 
        + \| \varrho_h (I_h^{\mathrm{RT}_0} \vecb{u}_h - \vecb{u}_h) \|_{L^1(T)} \right )  \\
        & \leq
        C \sum_T \| \vecb{q}_h - \varrho_h I_h^{\mathrm{RT}_0} \vecb{u}_h \|_{L^1(T)} + C \| \varrho_h \|_{L^2} \, \| I_h^{\mathrm{RT}_0} \vecb{u}_h
        -  \vecb{u}_h \|_{L^2}.
    \end{split}
\end{equation}
The term $\| \varrho_h \|_{L^2} \, \| I_h^{\mathrm{RT}_0} \vecb{u}_h
        -  \vecb{u}_h \|_{L^2}$ converges to $0$, according to
        the interpolation properties of $I_h^{\mathrm{RT}_0}$
        and the stability estimate for $\| \varrho_h \|_{L^2}$.
It remains to estimate $\sum_T \| \vecb{q}_h - \varrho_h I_h^{\mathrm{RT}_0} \vecb{u}_h \|_{L^1(T)}$.
Interpolation properties of $I_h^{\mathrm{RT}_0}$ yield
\begin{align*}
  \sum_T \| \vecb{q}_h - \varrho_h I_h^{\mathrm{RT}_0} \vecb{u}_h \|_{L^1(T)}
  & \lesssim \sum_T h_T \sum_{F \in \mathcal{F}(T)} \left \lvert (\varrho^\text{upw}_F - \varrho_h|_T) \int_F \vecb{u}_h \cdot \vecb{n}_F \, ds \right\rvert\\
  & \lesssim \sum_{F \in \mathcal{F}(\Omega)} h_F \lvert [[ \varrho_h ]]_F \rvert \, \lvert \vecb{u}_F \rvert := A.
\end{align*}
It holds \(A \rightarrow 0\) which can be proven as follows. A Cauchy inequality shows
\begin{align}\label{eqn:estimateA_case1}
  A \lesssim \left( \sum_{F \in \mathcal{F}(\Omega)}  \lvert \vecb{u}_F \rvert \varrho_{KL}^{-1} [[ \varrho_h ]]^2_F \right)^{1/2} \left(\sum_{F \in \mathcal{F}(\Omega)} h_F^2 \lvert \vecb{u}_F \rvert \varrho_{KL} \right)^{1/2}.
  \end{align}
The left sum is bounded by \eqref{eq:disc:stab:est:ib} and it remains to show that the second sum converges to zero. 
A H\"older inequality, a trace inequality and a inverse inequality on some
neighboring simplex \(T_F\) of \(F\) show
\begin{align}\label{eqn:bound_for_uF}
  \lvert \vecb{u}_F \rvert
  \lesssim \| \vecb{u} \|^{1/2}_{L^2(T_F)} \| \nabla \vecb{u} \|^{1/2}_{L^2(T_F)} \| 1 \|_{L^2(F)}
  \lesssim h_F^{d/2-1} \| \vecb{u} \|_{L^2(T_F)}.
\end{align}
Then, another Cauchy inequality and some overlap arguments yield
\begin{align*}
  \left(\sum_{F \in \mathcal{F}(\Omega)} h_F^2 \lvert u_F \rvert \varrho_{KL} \right)^{1/2}
  & \leq \left(\sum_{F \in \mathcal{F}(\Omega)} \| \vecb{u}\|_{L^2(T_F)}^2 \right)^{1/4} \left( \sum_{F \in \mathcal{F}(\Omega)} h_F^{d+2} \varrho_{KL}^2 \right)^{1/4}\\
  & \lesssim \| \nabla \vecb{u} \|^{1/2}_{L^2(\Omega)} \left( \sum_{F \in \mathcal{F}(\Omega)} h_F^{d+2} \varrho_{KL}^2 \right)^{1/4}
\end{align*}
To show that the last sum converges to zero, we use that \(\varrho_{KL}^2\) is smaller than \(\varrho_h|_{T_F}^2\) for some
neighboring simplex \(T_F\) of \(T\) and hence
\begin{align*}
 \left( \sum_{F \in \mathcal{F}(\Omega)} h_F^{d+2} \varrho_{KL}^2 \right)^{1/4}
 \lesssim \left( \sum_{F \in \mathcal{F}(\Omega)} h_F^2 \lvert T_F \rvert \varrho_h|_{T_F}^2 \right)^{1/4}
 \leq h^{1/2} \| \varrho_h \|_{L^2}^{1/2}.
\end{align*}
According to \eqref{eq:disc:stab:est:ii} $\| \varrho_h \|_{L^2}$ is bounded and hence, one arrives at
  \begin{align*}
    A \lesssim h^{1/2}
  \end{align*}
  which concludes the proof of Step 2.
  
  \medskip
  \textbf{Step 3. $(p,\varrho)$ satisfy the equation of state, i.e.\ \(p = \varphi(\varrho) = c \varrho\).}
  
  Consider any function \(\varphi \in C^\infty_C(\Omega)\) and its piecewise-constant approximation \(\varphi_h := \pi_0 \varphi\) which converges strongly to \(\varphi\). Then its holds
  \begin{align*}
    \int_\Omega p_h \varphi_h \, dx & \rightarrow \int_\Omega p \varphi \, dx\\
    \int_\Omega c \varrho_h \varphi_h \, dx & \rightarrow \int_\Omega c \varrho \varphi \, dx.
  \end{align*}
  Since the integrals on the right-hand side are equal for any \(h\), also their limit integrals have to be equal, i.e.
  \begin{align*}
    \int_\Omega p \varphi \, dx = \int_\Omega c \varrho \varphi \, dx \quad \text{for all } \varphi \in C^\infty_C(\Omega)
  \end{align*}
  Hence, it follows \(p = c \varrho\).
\end{proof}

\section{Numerical Experiments}\label{sec:numerics}
This section reports on some two-dimensional
numerical experiments assessing
accuracy and asymptotic convergence rates of the novel scheme,
which especially 
illustrates the increased robustness with respect to gradients in the momentum balance. Some experiments also show that the scheme might
also converge in barotropic cases where \(p = \varphi(\varrho) := c \varrho^\gamma\) with \(\gamma > 1\). In this case however the proof of Step 3 in the convergence proof (without using additional stability terms in the scheme) is non-trivial and open.

The loop in the algorithm was stopped in all experiments until the tolerance criterion was satisfied with \(\text{tol} := 10^{-11}\). The time step in the evolution of the density
was chosen small enough and usually \(\tau \approx \nu/c\). The term 'ndof'
refers to the number of degrees of freedom and is approximately
\(2 \mathcal{T} + \mathcal{N} + \mathcal{E}\), i.e. the size of the ansatz
spaces for density, pressure and velocity.

Throughout this section, $(\vecb{u}^+_h,\varrho^+_h)$ denotes the solution of the 'modified scheme' \eqref{eq:gradient:robust:scheme} and $(\vecb{u}_h,\varrho_h)$ denotes the solution of the 'classical scheme' \eqref{eq:gradient:robust:scheme} where \(\Pi = 1\).

\subsection{Manufactured solutions to study error convergence and locking}
This example on the unit square \(\Omega := (0,1)^2\) studies the convergence rates of our discretization scheme and examines the exact solution
\begin{align*}
  \vecb{u} := \mathrm{curl}(x^2(x-1)^2y^2(y-1)^2)/\varrho, \qquad p = \varphi(\varrho) := c\varrho^\gamma
\end{align*}
for different choices of \(\gamma\), \(\mu\) and \(\lambda = -2\mu/3\). 
Assuming a linear density \(\varrho := 1 + (y - 1/2)/c\),
the first test case considers the isothermal configuration
\(\gamma = 1\) and afterwards a barotropic configuration with
\(\gamma = 1.4\) is presented.
In all cases \(\int_\Omega \varrho \, dx = 1\) holds independent of \(c\).
The right-hand side functions
are chosen such that \((\vecb{u},p,\rho)\) is a solution of the compressible Stokes system with
\begin{align*}
  \vecb{f} := - 2 \mu \epsilon(\vecb{u}) - \frac{\mu}{3} \nabla(\mathrm{div} \vecb{u}) + \nabla p, \qquad
  \vecb{g} := 0.
\end{align*}
%or, marked by 'case with gravity' in the descriptions,
%\begin{align*}
%  \vecb{f} := - 2 \mu \epsilon(\vecb{u}) - \frac{\mu}{3} \nabla(\mathrm{div} \vecb{u}), \qquad
%  \vecb{g} := \nabla p / \varrho = 2 \nabla \varrho.
%\end{align*}

% macros for pgfplotstable
\pgfplotstableset{
    % #1 = row index
    % #2 = row style keys
    row style/.style 2 args={
        every row #1 column 1/.style={#2},
        every row #1 column 2/.style={#2},
        every row #1 column 3/.style={#2},
        every row #1 column 4/.style={#2},
        every row #1 column 5/.style={#2},
        every row #1 column 6/.style={#2},
        every row #1 column 7/.style={#2},
        every row #1 column 8/.style={#2},
        every row #1 column 9/.style={#2},
        every row #1 column 10/.style={#2},
    }
}

%%%%%%%%%%%%%%%%%%%%%%%%
% nu=1 c={1,100} %
%%%%%%%%%%%%%%%%%%%%%%%%

\pgfplotstableread[col sep=ampersand,row sep=\\]{
ndof & L2uBR+ &	 H1uBR+ & L2rhoBR+ & L2uBR &	 H1uBR & L2rhoBR \\
     161 & 1.6406e-03    & 5.3074e-02    & 6.3260e-02 & 2.2163e-03    & 5.0876e-02    & 6.6181e-02\\
     617 & 5.1687e-04    & 2.7185e-02    & 2.9715e-02 & 6.8682e-04    & 2.6798e-02    & 3.0472e-02\\
    2297 & 1.7243e-04    & 1.3593e-02    & 1.4988e-02 & 2.1981e-04    & 1.3299e-02    & 1.5277e-02\\
    9152 & 6.8132e-05    & 7.2351e-03    & 7.6828e-03 & 7.9905e-05    & 6.9205e-03    & 7.7950e-03\\
   36326 & 3.1398e-05    & 3.8502e-03    & 3.8844e-03 & 3.4143e-05    & 3.6769e-03    & 3.9254e-03\\
  143945 & 1.5167e-05    & 2.0803e-03    & 1.9910e-03 & 1.5772e-05    & 1.9997e-03    & 2.0053e-03\\
  573386 & 7.3871e-06    & 1.1250e-03    & 1.0212e-03 & 7.5349e-06    & 1.0883e-03    & 1.0256e-03\\
 2290184 & 3.6480e-06    & 6.0305e-04    & 5.2514e-04 & 3.6842e-06    & 5.8469e-04    & 5.2668e-04\\
}\tableAgOne %\tableESymGrad %nu=1 c=1 gamma=1

\pgfplotstableread[col sep=ampersand,row sep=\\]{
ndof & L2uBR+ &	 H1uBR+ & L2rhoBR+ & L2uBR &	 H1uBR & L2rhoBR \\
     161 & 1.4263e-03    & 4.2681e-02    & 6.2815e-04 & 1.9350e-03    & 4.2269e-02    & 6.5286e-04\\
     617 & 3.3655e-04    & 2.1185e-02    & 2.9560e-04 & 4.8204e-04    & 2.1516e-02    & 3.0179e-04\\
    2297 & 9.3756e-05    & 1.0842e-02    & 1.4850e-04 & 1.3713e-04    & 1.1034e-02    & 1.5111e-04\\
    9152 & 2.0188e-05    & 5.3259e-03    & 7.5477e-05 & 3.1425e-05    & 5.3609e-03    & 7.6773e-05\\
   36326 & 5.2149e-06    & 2.6664e-03    & 3.7654e-05 & 8.0579e-06    & 2.6846e-03    & 3.8197e-05\\
  143945 & 1.2726e-06    & 1.3288e-03    & 1.8908e-05 & 1.9892e-06    & 1.3413e-03    & 1.9156e-05\\
  573386 & 3.2005e-07    & 6.6387e-04    & 9.4417e-06 & 5.0723e-07    & 6.6994e-04    & 9.5619e-06\\
 2290184 & 8.5498e-08    & 3.3394e-04    & 4.7202e-06 & 1.3472e-07    & 3.3631e-04    & 4.7794e-06\\
}\tableBgOne %\tableFSymGrad %nu=1 c=100 gamma=1

\pgfplotstabletranspose[string type,
	create on use/Lrate+/.style={create col/dyadic refinement rate={L2uBR+}},
	create on use/Lrate/.style={create col/dyadic refinement rate={L2uBR}},
	create on use/Hrate+/.style={create col/dyadic refinement rate={H1uBR+}},
	create on use/Hrate/.style={create col/dyadic refinement rate={H1uBR}},
	create on use/Drate+/.style={create col/dyadic refinement rate={L2rhoBR+}},
	create on use/Drate/.style={create col/dyadic refinement rate={L2rhoBR}},
    columns={ndof,L2uBR+,Lrate+,L2uBR,Lrate,H1uBR+,Hrate+,H1uBR,Hrate,L2rhoBR+,Drate+,L2rhoBR,Drate},
    colnames from=ndof,
    input colnames to=ndof
]\tableAgOnetranspose{\tableAgOne}

\pgfplotstabletranspose[string type,
	create on use/Lrate+/.style={create col/dyadic refinement rate={L2uBR+}},
	create on use/Lrate/.style={create col/dyadic refinement rate={L2uBR}},
	create on use/Hrate+/.style={create col/dyadic refinement rate={H1uBR+}},
	create on use/Hrate/.style={create col/dyadic refinement rate={H1uBR}},
	create on use/Drate+/.style={create col/dyadic refinement rate={L2rhoBR+}},
	create on use/Drate/.style={create col/dyadic refinement rate={L2rhoBR}},
    columns={ndof,L2uBR+,Lrate+,L2uBR,Lrate,H1uBR+,Hrate+,H1uBR,Hrate,L2rhoBR+,Drate+,L2rhoBR,Drate},
    colnames from=ndof,
    input colnames to=ndof
]\tableBgOnetranspose{\tableBgOne}

%%%%%%%%%%%%%%%%%%%%%%%%
% nu=1e-1 c={1,100} %
%%%%%%%%%%%%%%%%%%%%%%%%

\pgfplotstableread[col sep=ampersand,row sep=\\]{
ndof & L2uBR+ &	 H1uBR+ & L2rhoBR+ & L2uBR &	 H1uBR & L2rhoBR \\
     161 & 1.6492e-03    & 5.3035e-02    & 6.2509e-02 & 8.6276e-03    & 1.3289e-01    & 6.4614e-02\\
     617 & 5.2150e-04    & 2.7223e-02    & 2.9450e-02 & 2.4338e-03    & 7.7364e-02    & 3.0112e-02\\
    2297 & 1.7880e-04    & 1.4004e-02    & 1.4789e-02 & 6.8645e-04    & 4.0800e-02    & 1.5047e-02\\
    9152 & 7.2297e-05    & 7.8991e-03    & 7.5198e-03 & 1.9937e-04    & 2.1302e-02    & 7.6425e-03\\
   36326 & 3.3919e-05    & 5.1988e-03    & 3.7541e-03 & 6.3455e-05    & 1.1166e-02    & 3.8076e-03\\
  143945 & 1.6517e-05    & 3.9150e-03    & 1.8906e-03 & 2.3100e-05    & 6.3231e-03    & 1.9148e-03\\
  573386 & 8.0075e-06    & 3.0203e-03    & 9.5060e-04 & 9.5158e-06    & 3.8831e-03    & 9.6114e-04\\
 2290184 & 3.8916e-06    & 2.1969e-03    & 4.8150e-04 & 4.2498e-06    & 2.5115e-03    & 4.8588e-04\\
}\tableCgOne %\tableESymGrad %nu=1e-1 c=1 gamma=1

\pgfplotstableread[col sep=ampersand,row sep=\\]{
ndof & L2uBR+ &	 H1uBR+ & L2rhoBR+ & L2uBR &	 H1uBR & L2rhoBR \\
     161 & 1.4264e-03    & 4.2681e-02    & 6.2503e-04 & 8.3545e-03    & 1.3101e-01    & 6.4507e-04\\
     617 & 3.3656e-04    & 2.1185e-02    & 2.9448e-04 & 2.2750e-03    & 7.6136e-02    & 3.0097e-04\\
    2297 & 9.3757e-05    & 1.0842e-02    & 1.4786e-04 & 6.1015e-04    & 4.0233e-02    & 1.5041e-04\\
    9152 & 2.0188e-05    & 5.3256e-03    & 7.5161e-05 & 1.5620e-04    & 2.0806e-02    & 7.6403e-05\\
   36326 & 5.2150e-06    & 2.6661e-03    & 3.7485e-05 & 3.9872e-05    & 1.0442e-02    & 3.8045e-05\\
  143945 & 1.2726e-06    & 1.3285e-03    & 1.8826e-05 & 1.0069e-05    & 5.3120e-03    & 1.9094e-05\\
  573386 & 3.2009e-07    & 6.6366e-04    & 9.4010e-06 & 2.5413e-06    & 2.6500e-03    & 9.5324e-06\\
 2290184 & 8.5601e-08    & 3.3391e-04    & 4.6988e-06 & 6.4652e-07    & 1.3303e-03    & 4.7639e-06\\
}\tableDgOne %\tableFSymGrad %nu=1e-1 c=100 gamma=1

\pgfplotstabletranspose[string type,
	create on use/Lrate+/.style={create col/dyadic refinement rate={L2uBR+}},
	create on use/Lrate/.style={create col/dyadic refinement rate={L2uBR}},
	create on use/Hrate+/.style={create col/dyadic refinement rate={H1uBR+}},
	create on use/Hrate/.style={create col/dyadic refinement rate={H1uBR}},
	create on use/Drate+/.style={create col/dyadic refinement rate={L2rhoBR+}},
	create on use/Drate/.style={create col/dyadic refinement rate={L2rhoBR}},
    columns={ndof,L2uBR+,Lrate+,L2uBR,Lrate,H1uBR+,Hrate+,H1uBR,Hrate,L2rhoBR+,Drate+,L2rhoBR,Drate},
    colnames from=ndof,
    input colnames to=ndof
]\tableCgOnetranspose{\tableCgOne}

\pgfplotstabletranspose[string type,
	create on use/Lrate+/.style={create col/dyadic refinement rate={L2uBR+}},
	create on use/Lrate/.style={create col/dyadic refinement rate={L2uBR}},
	create on use/Hrate+/.style={create col/dyadic refinement rate={H1uBR+}},
	create on use/Hrate/.style={create col/dyadic refinement rate={H1uBR}},
	create on use/Drate+/.style={create col/dyadic refinement rate={L2rhoBR+}},
	create on use/Drate/.style={create col/dyadic refinement rate={L2rhoBR}},
    columns={ndof,L2uBR+,Lrate+,L2uBR,Lrate,H1uBR+,Hrate+,H1uBR,Hrate,L2rhoBR+,Drate+,L2rhoBR,Drate},
    colnames from=ndof,
    input colnames to=ndof
]\tableDgOnetranspose{\tableDgOne}

%%%%%%%%%%%%%%%%%%%%%%%%
% nu=10^{-2} c={1,100} %
%%%%%%%%%%%%%%%%%%%%%%%%

\pgfplotstableread[col sep=ampersand,row sep=\\]{
ndof & L2uBR+ &	 H1uBR+ & L2rhoBR+ & L2uBR &	 H1uBR & L2rhoBR \\
     161 & 1.6503e-03    & 5.3032e-02    & 6.2500e-02 & 8.4443e-02    & 1.2485e+00    & 6.4516e-02\\
     617 & 5.2211e-04    & 2.7232e-02    & 2.9446e-02 & 2.2952e-02    & 7.3069e-01    & 3.0104e-02\\
    2297 & 1.7975e-04    & 1.4102e-02    & 1.4785e-02 & 6.0453e-03    & 3.8740e-01    & 1.5041e-02\\
    9152 & 7.3108e-05    & 8.0909e-03    & 7.5158e-03 & 1.5683e-03    & 2.0237e-01    & 7.6391e-03\\
   36326 & 3.4655e-05    & 5.7504e-03    & 3.7484e-03 & 4.0991e-04    & 1.0153e-01    & 3.8041e-03\\
  143945 & 1.7128e-05    & 4.9715e-03    & 1.8827e-03 & 1.0825e-04    & 5.1790e-02    & 1.9093e-03\\
  573386 & 8.4708e-06    & 4.6672e-03    & 9.4032e-04 & 2.9841e-05    & 2.6074e-02    & 9.5323e-04\\
 2290184 & 4.1896e-06    & 4.3387e-03    & 4.7030e-04 & 9.0358e-06    & 1.3544e-02    & 4.7657e-04\\
}\tableEgOne %\tableESymGrad %nu=1e-2 c=1 gamma=1

\pgfplotstableread[col sep=ampersand,row sep=\\]{
ndof & L2uBR+ &	 H1uBR+ & L2rhoBR+ & L2uBR &	 H1uBR & L2rhoBR \\
     161 & 1.4264e-03    & 4.2681e-02    & 6.2500e-04 & 8.2849e-02    & 1.2507e+00    & 6.4458e-04\\
     617 & 3.3656e-04    & 2.1185e-02    & 2.9446e-04 & 2.2552e-02    & 7.3145e-01    & 3.0099e-04\\
    2297 & 9.3757e-05    & 1.0842e-02    & 1.4785e-04 & 5.9433e-03    & 3.8757e-01    & 1.5040e-04\\
    9152 & 2.0188e-05    & 5.3256e-03    & 7.5158e-05 & 1.5293e-03    & 2.0253e-01    & 7.6391e-05\\
   36326 & 5.2150e-06    & 2.6661e-03    & 3.7484e-05 & 3.9081e-04    & 1.0155e-01    & 3.8043e-05\\
  143945 & 1.2726e-06    & 1.3285e-03    & 1.8826e-05 & 9.8718e-05    & 5.1688e-02    & 1.9093e-05\\
  573386 & 3.2009e-07    & 6.6364e-04    & 9.4006e-06 & 2.4747e-05    & 2.5770e-02    & 9.5324e-06\\
 2290184 & 8.5614e-08    & 3.3392e-04    & 4.6986e-06 & 6.2117e-06    & 1.2941e-02    & 4.7639e-06\\
}\tableFgOne %\tableFSymGrad %nu=1e-2 c=100 gamma=1

\pgfplotstabletranspose[string type,
	create on use/Lrate+/.style={create col/dyadic refinement rate={L2uBR+}},
	create on use/Lrate/.style={create col/dyadic refinement rate={L2uBR}},
	create on use/Hrate+/.style={create col/dyadic refinement rate={H1uBR+}},
	create on use/Hrate/.style={create col/dyadic refinement rate={H1uBR}},
	create on use/Drate+/.style={create col/dyadic refinement rate={L2rhoBR+}},
	create on use/Drate/.style={create col/dyadic refinement rate={L2rhoBR}},
    columns={ndof,L2uBR+,Lrate+,L2uBR,Lrate,H1uBR+,Hrate+,H1uBR,Hrate,L2rhoBR+,Drate+,L2rhoBR,Drate},
    colnames from=ndof,
    input colnames to=ndof
]\tableEgOnetranspose{\tableEgOne}

\pgfplotstabletranspose[string type,
	create on use/Lrate+/.style={create col/dyadic refinement rate={L2uBR+}},
	create on use/Lrate/.style={create col/dyadic refinement rate={L2uBR}},
	create on use/Hrate+/.style={create col/dyadic refinement rate={H1uBR+}},
	create on use/Hrate/.style={create col/dyadic refinement rate={H1uBR}},
	create on use/Drate+/.style={create col/dyadic refinement rate={L2rhoBR+}},
	create on use/Drate/.style={create col/dyadic refinement rate={L2rhoBR}},
    columns={ndof,L2uBR+,Lrate+,L2uBR,Lrate,H1uBR+,Hrate+,H1uBR,Hrate,L2rhoBR+,Drate+,L2rhoBR,Drate},
    colnames from=ndof,
    input colnames to=ndof
]\tableFgOnetranspose{\tableFgOne}

%%%%%%%%%%%%%%%%%%%%%%%%
% nu=10^{-4} c={1,100} %
%%%%%%%%%%%%%%%%%%%%%%%%

\pgfplotstableread[col sep=ampersand,row sep=\\]{
ndof & L2uBR+ &	 H1uBR+ & L2rhoBR+ & L2uBR &	 H1uBR & L2rhoBR \\
     161 & 1.6504e-03    & 5.3032e-02    & 6.2500e-02 & 8.4598e+00    & 1.2494e+02    & 6.4510e-02\\
     617 & 5.2217e-04    & 2.7233e-02    & 2.9446e-02 & 2.2941e+00    & 7.3032e+01    & 3.0105e-02\\
    2297 & 1.7986e-04    & 1.4114e-02    & 1.4785e-02 & 6.0013e-01    & 3.8723e+01    & 1.5041e-02\\
    9152 & 7.3207e-05    & 8.1159e-03    & 7.5158e-03 & 1.5365e-01    & 2.0252e+01    & 7.6391e-03\\
   36326 & 3.4755e-05    & 5.8302e-03    & 3.7484e-03 & 3.9263e-02    & 1.0153e+01    & 3.8043e-03\\
  143945 & 1.7225e-05    & 5.1484e-03    & 1.8826e-03 & 9.9078e-03    & 5.1674e+00    & 1.9093e-03\\
  573386 & 8.5658e-06    & 5.0214e-03    & 9.4006e-04 & 2.4836e-03    & 2.5762e+00    & 9.5324e-04\\
 2290184 & 4.2762e-06    & 4.9878e-03    & 4.6986e-04 & 6.2346e-04    & 1.2937e+00    & 4.7639e-04\\
}\tableGgOne %\tableESymGrad %nu=1e-4 c=1 gamma=1

\pgfplotstableread[col sep=ampersand,row sep=\\]{
ndof & L2uBR+ &	 H1uBR+ & L2rhoBR+ & L2uBR &	 H1uBR & L2rhoBR \\
     161 & 1.4264e-03    & 4.2681e-02    & 6.2500e-04 & 8.2969e+00    & 1.2505e+02    & 6.4453e-04\\
     617 & 3.3656e-04    & 2.1185e-02    & 2.9446e-04 & 2.2572e+00    & 7.3083e+01    & 3.0099e-04\\
    2297 & 9.3757e-05    & 1.0842e-02    & 1.4785e-04 & 5.9381e-01    & 3.8725e+01    & 1.5040e-04\\
    9152 & 2.0188e-05    & 5.3256e-03    & 7.5158e-05 & 1.5281e-01    & 2.0252e+01    & 7.6390e-05\\
   36326 & 5.2150e-06    & 2.6661e-03    & 3.7484e-05 & 3.9054e-02    & 1.0153e+01    & 3.8043e-05\\
  143945 & 1.2726e-06    & 1.3285e-03    & 1.8826e-05 & 9.8643e-03    & 5.1675e+00    & 1.9093e-05\\
  573386 & 3.2009e-07    & 6.6364e-04    & 9.4006e-06 & 2.4708e-03    & 2.5762e+00    & 9.5324e-06\\
 2290184 & 8.5605e-08    & 3.3392e-04    & 4.6986e-06 & 6.1928e-04    & 1.2938e+00    & 4.7640e-06\\
}\tableHgOne %\tableFSymGrad %nu=1e-4 c=100 gamma=1

\pgfplotstabletranspose[string type,
	create on use/Lrate+/.style={create col/dyadic refinement rate={L2uBR+}},
	create on use/Lrate/.style={create col/dyadic refinement rate={L2uBR}},
	create on use/Hrate+/.style={create col/dyadic refinement rate={H1uBR+}},
	create on use/Hrate/.style={create col/dyadic refinement rate={H1uBR}},
	create on use/Drate+/.style={create col/dyadic refinement rate={L2rhoBR+}},
	create on use/Drate/.style={create col/dyadic refinement rate={L2rhoBR}},
    columns={ndof,L2uBR+,Lrate+,L2uBR,Lrate,H1uBR+,Hrate+,H1uBR,Hrate,L2rhoBR+,Drate+,L2rhoBR,Drate},
    colnames from=ndof,
    input colnames to=ndof
]\tableGgOnetranspose{\tableGgOne}

\pgfplotstabletranspose[string type,
	create on use/Lrate+/.style={create col/dyadic refinement rate={L2uBR+}},
	create on use/Lrate/.style={create col/dyadic refinement rate={L2uBR}},
	create on use/Hrate+/.style={create col/dyadic refinement rate={H1uBR+}},
	create on use/Hrate/.style={create col/dyadic refinement rate={H1uBR}},
	create on use/Drate+/.style={create col/dyadic refinement rate={L2rhoBR+}},
	create on use/Drate/.style={create col/dyadic refinement rate={L2rhoBR}},
    columns={ndof,L2uBR+,Lrate+,L2uBR,Lrate,H1uBR+,Hrate+,H1uBR,Hrate,L2rhoBR+,Drate+,L2rhoBR,Drate},
    colnames from=ndof,
    input colnames to=ndof
]\tableHgOnetranspose{\tableHgOne}

%%%%%%%%
%FIGURES
%%%%%%%%

\begin{figure}
\centering
\begin{tikzpicture}
\begin{axis}[
         ymin=1e-8,
         ymax=1e-1,
         xmin=1e2,
         xmax=1e7,
        xmode=log,
        ymode=log,
        xlabel near ticks,
        ylabel near ticks,
        xlabel=ndof,
        width=0.405\textwidth,
        height=175pt,
       title=,
       legend style={at={(1.1,0.0)},anchor=south west,legend cell align=left}]
\addplot[color=blue,mark=o] table[x index=0, y index=1]{\tableAgOne};    
\addlegendentry{\tiny $\| \vecb{u} - \vecb{u}^+_h \|_{L^2}$}
\addplot[color=red,mark=o] table[x index=0, y index=4]{\tableAgOne};   
\addlegendentry{\tiny $\| \vecb{u} - \vecb{u}_h \|_{L^2}$}

\addplot[color=blue,mark=*] table[x index=0, y index=2]{\tableAgOne};    
\addlegendentry{\tiny $\| \nabla(\vecb{u} - \vecb{u}^+_h) \|_{L^2}$}
\addplot[color=red,mark=*] table[x index=0, y index=5]{\tableAgOne};     
\addlegendentry{\tiny $\| \nabla(\vecb{u} - \vecb{u}_h) \|_{L^2}$}

\addplot[color=blue,mark=x] table[x index=0, y index=3]{\tableAgOne};    
\addlegendentry{\tiny $\| \varrho - \varrho^+_h \|_{L^2}$}
\addplot[color=red,mark=x] table[x index=0, y index=6]{\tableAgOne};     
\addlegendentry{\tiny $\| \varrho - \varrho_h \|_{L^2}$}

\addplot[domain=1e0:1e7, color=gray, dashed]{0.4/pow(x,0.5)};
\addlegendentry{\tiny $h$}
\addplot[domain=1e0:1e7, color=gray,dotted]{1e-1/pow(x,1)};
\addlegendentry{\tiny $h^2$}
\end{axis}
\end{tikzpicture}
\hfill
\begin{tikzpicture}
\begin{axis}[
         ymin=1e-8,
         ymax=1e-1,
         xmin=1e2,
         xmax=1e7,
        xmode=log,
        ymode=log,
        xlabel near ticks,
        ylabel near ticks,
        xlabel=ndof,
        width=0.405\textwidth,
        height=175pt,
       title=,]
\addplot[color=blue,mark=o] table[x index=0, y index=1]{\tableBgOne}; 
\addplot[color=red,mark=o] table[x index=0, y index=4]{\tableBgOne};  
  
\addplot[color=blue,mark=*] table[x index=0, y index=2]{\tableBgOne}; 
\addplot[color=red,mark=*] table[x index=0, y index=5]{\tableBgOne}; 
   
\addplot[color=blue,mark=x] table[x index=0, y index=3]{\tableBgOne};  
\addplot[color=red,mark=x] table[x index=0, y index=6]{\tableBgOne}; 

\addplot[domain=1e0:1e7, color=gray, dashed]{0.4/pow(x,0.5)};
\addplot[domain=1e0:1e7, color=gray,dotted]{1e-1/pow(x,1)};
\end{axis}
\end{tikzpicture}

\caption{\label{fig:resultsA_mu1}
Convergence histories for the modified method and classical method for $\gamma=1$ and $\mu = 1$ and $c=1$ (left), $c=100$ (right).}
\end{figure}

\begin{figure}
\centering
\begin{tikzpicture}
\begin{axis}[
         ymin=1e-8,
         ymax=1e0,
         xmin=1e2,
         xmax=1e7,
        xmode=log,
        ymode=log,
        xlabel near ticks,
        ylabel near ticks,
        xlabel=ndof,
        width=0.405\textwidth,
        height=175pt,
       title=,
       legend style={at={(1.1,0.0)},anchor=south west,legend cell align=left}]
\addplot[color=blue,mark=o] table[x index=0, y index=1]{\tableCgOne};    
\addlegendentry{\tiny $\| \vecb{u} - \vecb{u}^+_h \|_{L^2}$}
\addplot[color=red,mark=o] table[x index=0, y index=4]{\tableCgOne};   
\addlegendentry{\tiny $\| \vecb{u} - \vecb{u}_h \|_{L^2}$}

\addplot[color=blue,mark=*] table[x index=0, y index=2]{\tableCgOne};    
\addlegendentry{\tiny $\| \nabla(\vecb{u} - \vecb{u}^+_h) \|_{L^2}$}
\addplot[color=red,mark=*] table[x index=0, y index=5]{\tableCgOne};     
\addlegendentry{\tiny $\| \nabla(\vecb{u} - \vecb{u}_h) \|_{L^2}$}

\addplot[color=blue,mark=x] table[x index=0, y index=3]{\tableCgOne};    
\addlegendentry{\tiny $\| \varrho - \varrho^+_h \|_{L^2}$}
\addplot[color=red,mark=x] table[x index=0, y index=6]{\tableCgOne};     
\addlegendentry{\tiny $\| \varrho - \varrho_h \|_{L^2}$}

\addplot[domain=1e0:1e7, color=gray, dashed]{0.4/pow(x,0.5)};
\addlegendentry{\tiny $h$}
\addplot[domain=1e0:1e7, color=gray,dotted]{1e-1/pow(x,1)};
\addlegendentry{\tiny $h^2$}
\end{axis}
\end{tikzpicture}
\hfill
\begin{tikzpicture}
\begin{axis}[
         ymin=1e-8,
         ymax=1e0,
         xmin=1e2,
         xmax=1e7,
        xmode=log,
        ymode=log,
        xlabel near ticks,
        ylabel near ticks,
        xlabel=ndof,
        width=0.405\textwidth,
        height=175pt,
       title=,]
\addplot[color=blue,mark=o] table[x index=0, y index=1]{\tableDgOne}; 
\addplot[color=red,mark=o] table[x index=0, y index=4]{\tableDgOne};  
  
\addplot[color=blue,mark=*] table[x index=0, y index=2]{\tableDgOne}; 
\addplot[color=red,mark=*] table[x index=0, y index=5]{\tableDgOne}; 
   
\addplot[color=blue,mark=x] table[x index=0, y index=3]{\tableDgOne};  
\addplot[color=red,mark=x] table[x index=0, y index=6]{\tableDgOne}; 

\addplot[domain=1e0:1e7, color=gray, dashed]{0.4/pow(x,0.5)};
\addplot[domain=1e0:1e7, color=gray,dotted]{1e-1/pow(x,1)};
\end{axis}
\end{tikzpicture}

\caption{\label{fig:resultsA_mu10}
Convergence histories for the modified method and classical method for $\gamma=1$ and $\mu = 10^{-1}$ and $c=1$ (left), $c=100$ (right).}
\end{figure}

\begin{figure}
\centering
\begin{tikzpicture}
\begin{axis}[
         ymin=1e-8,
         ymax=1e1,
         xmin=1e2,
         xmax=1e7,
        xmode=log,
        ymode=log,
        xlabel near ticks,
        ylabel near ticks,
        xlabel=ndof,
        width=0.405\textwidth,
        height=175pt,
       title=,
       legend style={at={(1.1,0.0)},anchor=south west,legend cell align=left}]
\addplot[color=blue,mark=o] table[x index=0, y index=1]{\tableEgOne};    
\addlegendentry{\tiny $\| \vecb{u} - \vecb{u}^+_h \|_{L^2}$}
\addplot[color=red,mark=o] table[x index=0, y index=4]{\tableEgOne};   
\addlegendentry{\tiny $\| \vecb{u} - \vecb{u}_h \|_{L^2}$}

\addplot[color=blue,mark=*] table[x index=0, y index=2]{\tableEgOne};    
\addlegendentry{\tiny $\| \nabla(\vecb{u} - \vecb{u}^+_h) \|_{L^2}$}
\addplot[color=red,mark=*] table[x index=0, y index=5]{\tableEgOne};     
\addlegendentry{\tiny $\| \nabla(\vecb{u} - \vecb{u}_h) \|_{L^2}$}

\addplot[color=blue,mark=x] table[x index=0, y index=3]{\tableEgOne};    
\addlegendentry{\tiny $\| \varrho - \varrho^+_h \|_{L^2}$}
\addplot[color=red,mark=x] table[x index=0, y index=6]{\tableEgOne};     
\addlegendentry{\tiny $\| \varrho - \varrho_h \|_{L^2}$}

\addplot[domain=1e0:1e7, color=gray, dashed]{0.4/pow(x,0.5)};
\addlegendentry{\tiny $h$}
\addplot[domain=1e0:1e7, color=gray,dotted]{1e-1/pow(x,1)};
\addlegendentry{\tiny $h^2$}
\end{axis}
\end{tikzpicture}
\hfill
\begin{tikzpicture}
\begin{axis}[
         ymin=1e-8,
         ymax=1e1,
         xmin=1e2,
         xmax=1e7,
        xmode=log,
        ymode=log,
        xlabel near ticks,
        ylabel near ticks,
        xlabel=ndof,
        width=0.405\textwidth,
        height=175pt,
       title=,]
\addplot[color=blue,mark=o] table[x index=0, y index=1]{\tableFgOne}; 
\addplot[color=red,mark=o] table[x index=0, y index=4]{\tableFgOne};  
  
\addplot[color=blue,mark=*] table[x index=0, y index=2]{\tableFgOne}; 
\addplot[color=red,mark=*] table[x index=0, y index=5]{\tableFgOne}; 
   
\addplot[color=blue,mark=x] table[x index=0, y index=3]{\tableFgOne};  
\addplot[color=red,mark=x] table[x index=0, y index=6]{\tableFgOne}; 

\addplot[domain=1e0:1e7, color=gray, dashed]{0.4/pow(x,0.5)};
\addplot[domain=1e0:1e7, color=gray,dotted]{1e-1/pow(x,1)};
\end{axis}
\end{tikzpicture}

\caption{\label{fig:resultsA_mu100}
Convergence histories for the modified method and classical method for $\gamma=1$ and $\mu = 10^{-2}$ and $c=1$ (left), $c=100$ (right).}
\end{figure}

\begin{figure}
\centering
\begin{tikzpicture}
\begin{axis}[
         ymin=1e-8,
         ymax=1e3,
         xmin=1e2,
         xmax=1e7,
        xmode=log,
        ymode=log,
        xlabel near ticks,
        ylabel near ticks,
        xlabel=ndof,
        width=0.405\textwidth,
        height=175pt,
       title=,
       legend style={at={(1.1,0.0)},anchor=south west,legend cell align=left}]
\addplot[color=blue,mark=o] table[x index=0, y index=1]{\tableGgOne};    
\addlegendentry{\tiny $\| \vecb{u} - \vecb{u}^+_h \|_{L^2}$}
\addplot[color=red,mark=o] table[x index=0, y index=4]{\tableGgOne};   
\addlegendentry{\tiny $\| \vecb{u} - \vecb{u}_h \|_{L^2}$}

\addplot[color=blue,mark=*] table[x index=0, y index=2]{\tableGgOne};    
\addlegendentry{\tiny $\| \nabla(\vecb{u} - \vecb{u}^+_h) \|_{L^2}$}
\addplot[color=red,mark=*] table[x index=0, y index=5]{\tableGgOne};     
\addlegendentry{\tiny $\| \nabla(\vecb{u} - \vecb{u}_h) \|_{L^2}$}

\addplot[color=blue,mark=x] table[x index=0, y index=3]{\tableGgOne};    
\addlegendentry{\tiny $\| \varrho - \varrho^+_h \|_{L^2}$}
\addplot[color=red,mark=x] table[x index=0, y index=6]{\tableGgOne};     
\addlegendentry{\tiny $\| \varrho - \varrho_h \|_{L^2}$}

\addplot[domain=1e0:1e7, color=gray, dashed]{0.4/pow(x,0.5)};
\addlegendentry{\tiny $h$}
\addplot[domain=1e0:1e7, color=gray,dotted]{1e-1/pow(x,1)};
\addlegendentry{\tiny $h^2$}
\end{axis}
\end{tikzpicture}
\hfill
\begin{tikzpicture}
\begin{axis}[
         ymin=1e-8,
         ymax=1e3,
         xmin=1e2,
         xmax=1e7,
        xmode=log,
        ymode=log,
        xlabel near ticks,
        ylabel near ticks,
        xlabel=ndof,
        width=0.405\textwidth,
        height=175pt,
       title=,]
\addplot[color=blue,mark=o] table[x index=0, y index=1]{\tableHgOne}; 
\addplot[color=red,mark=o] table[x index=0, y index=4]{\tableHgOne};  
  
\addplot[color=blue,mark=*] table[x index=0, y index=2]{\tableHgOne}; 
\addplot[color=red,mark=*] table[x index=0, y index=5]{\tableHgOne}; 
   
\addplot[color=blue,mark=x] table[x index=0, y index=3]{\tableHgOne};  
\addplot[color=red,mark=x] table[x index=0, y index=6]{\tableHgOne}; 

\addplot[domain=1e0:1e7, color=gray, dashed]{0.4/pow(x,0.5)};
\addplot[domain=1e0:1e7, color=gray,dotted]{1e-1/pow(x,1)};
\end{axis}
\end{tikzpicture}

\caption{\label{fig:resultsA_mu10000}
Convergence histories for the modified method and classical method for $\gamma=1$ and $\mu = 10^{-4}$ and $c=1$ (left), $c=100$ (right).}
\end{figure}

%%%%%%%
%TABLES
%%%%%%%

\begin{sidewaystable}
\centering
$\gamma=1 \qquad \mu=1 \qquad c=1$ \\
\pgfplotstabletypeset[columns={ndof,161,617,2297,9152,36326,143945,573386,2290184},
precision=2,sci,sci 10e,sci zerofill,
  columns/ndof/.style={string type,
  string replace={Drate+}{rate},
  string replace={Hrate+}{rate},
  string replace={Lrate+}{rate},
  string replace={Drate}{rate},
  string replace={Hrate}{rate},
  string replace={Lrate}{rate},
  string replace={L2uBR}{$\| \vecb{u} - \vecb{u}_h\|_{L^2} $},
  string replace={L2uBR+}{$\| \vecb{u} - \vecb{u}^+_h\|_{L^2} $},
  string replace={H1uBR}{$\| \nabla(\vecb{u} - \vecb{u}_h)\|_{L^2} $},
  string replace={H1uBR+}{$\| \nabla(\vecb{u} - \vecb{u}^+_h)\|_{L^2} $},
  string replace={L2rhoBR}{$\| \varrho - \varrho_h\|_{L^2} $},
  string replace={L2rhoBR+}{$\| \varrho - \varrho^+_h\|_{L^2} $},
  },
row style={1}{fixed,precision=2,zerofill},
row style={3}{fixed,precision=2,zerofill},
row style={5}{fixed,precision=2,zerofill},
row style={7}{fixed,precision=2,zerofill},
row style={9}{fixed,precision=2,zerofill},
row style={11}{fixed,precision=2,zerofill},
every head row/.style={before row={\toprule}},
every row no 0/.style={before row=\midrule},
every row no 4/.style={before row=\midrule},
every row no 8/.style={before row=\midrule},
every last row/.style={after row={\bottomrule}},
]\tableAgOnetranspose
\vspace{2ex}

$\gamma=1 \qquad \mu=1 \qquad c=100$ \\
\pgfplotstabletypeset[columns={ndof,161,617,2297,9152,36326,143945,573386,2290184},
precision=2,sci,sci 10e,sci zerofill,
  columns/ndof/.style={string type,
  string replace={Drate+}{rate},
  string replace={Hrate+}{rate},
  string replace={Lrate+}{rate},
  string replace={Drate}{rate},
  string replace={Hrate}{rate},
  string replace={Lrate}{rate},
  string replace={L2uBR}{$\| \vecb{u} - \vecb{u}_h\|_{L^2} $},
  string replace={L2uBR+}{$\| \vecb{u} - \vecb{u}^+_h\|_{L^2} $},
  string replace={H1uBR}{$\| \nabla(\vecb{u} - \vecb{u}_h)\|_{L^2} $},
  string replace={H1uBR+}{$\| \nabla(\vecb{u} - \vecb{u}^+_h)\|_{L^2} $},
  string replace={L2rhoBR}{$\| \varrho - \varrho_h\|_{L^2} $},
  string replace={L2rhoBR+}{$\| \varrho - \varrho^+_h\|_{L^2} $},
  },
row style={1}{fixed,precision=2,zerofill},
row style={3}{fixed,precision=2,zerofill},
row style={5}{fixed,precision=2,zerofill},
row style={7}{fixed,precision=2,zerofill},
row style={9}{fixed,precision=2,zerofill},
row style={11}{fixed,precision=2,zerofill},
every head row/.style={before row={\toprule}},
every row no 0/.style={before row=\midrule},
every row no 4/.style={before row=\midrule},
every row no 8/.style={before row=\midrule},
every last row/.style={after row={\bottomrule}},
]\tableBgOnetranspose
\caption{\label{tab:resultsA_mu1}Errors of the modified gradient-robust scheme $(\vecb{u}_h^+,\rho^+)$ and the classical scheme $(\vecb{u}_h,\rho)$ for \(\gamma=1\) and \(\mu=1\).}
\end{sidewaystable}

\begin{sidewaystable}
\centering
$\gamma=1 \qquad \mu=10^{-2} \qquad c=1$ \\
\pgfplotstabletypeset[columns={ndof,161,617,2297,9152,36326,143945,573386,2290184},
precision=2,sci,sci 10e,sci zerofill,
  columns/ndof/.style={string type,
  string replace={Drate+}{rate},
  string replace={Hrate+}{rate},
  string replace={Lrate+}{rate},
  string replace={Drate}{rate},
  string replace={Hrate}{rate},
  string replace={Lrate}{rate},
  string replace={L2uBR}{$\| \vecb{u} - \vecb{u}_h\|_{L^2} $},
  string replace={L2uBR+}{$\| \vecb{u} - \vecb{u}^+_h\|_{L^2} $},
  string replace={H1uBR}{$\| \nabla(\vecb{u} - \vecb{u}_h)\|_{L^2} $},
  string replace={H1uBR+}{$\| \nabla(\vecb{u} - \vecb{u}^+_h)\|_{L^2} $},
  string replace={L2rhoBR}{$\| \varrho - \varrho_h\|_{L^2} $},
  string replace={L2rhoBR+}{$\| \varrho - \varrho^+_h\|_{L^2} $},
  },
row style={1}{fixed,precision=2,zerofill},
row style={3}{fixed,precision=2,zerofill},
row style={5}{fixed,precision=2,zerofill},
row style={7}{fixed,precision=2,zerofill},
row style={9}{fixed,precision=2,zerofill},
row style={11}{fixed,precision=2,zerofill},
every head row/.style={before row={\toprule}},
every row no 0/.style={before row=\midrule},
every row no 4/.style={before row=\midrule},
every row no 8/.style={before row=\midrule},
every last row/.style={after row={\bottomrule}},
]\tableEgOnetranspose
\vspace{2ex}

$\gamma=1 \qquad \mu=10^{-2} \qquad c=100$ \\
\pgfplotstabletypeset[columns={ndof,161,617,2297,9152,36326,143945,573386,2290184},
precision=2,sci,sci 10e,sci zerofill,
  columns/ndof/.style={string type,
  string replace={Drate+}{rate},
  string replace={Hrate+}{rate},
  string replace={Lrate+}{rate},
  string replace={Drate}{rate},
  string replace={Hrate}{rate},
  string replace={Lrate}{rate},
  string replace={L2uBR}{$\| \vecb{u} - \vecb{u}_h\|_{L^2} $},
  string replace={L2uBR+}{$\| \vecb{u} - \vecb{u}^+_h\|_{L^2} $},
  string replace={H1uBR}{$\| \nabla(\vecb{u} - \vecb{u}_h)\|_{L^2} $},
  string replace={H1uBR+}{$\| \nabla(\vecb{u} - \vecb{u}^+_h)\|_{L^2} $},
  string replace={L2rhoBR}{$\| \varrho - \varrho_h\|_{L^2} $},
  string replace={L2rhoBR+}{$\| \varrho - \varrho^+_h\|_{L^2} $},
  },
row style={1}{fixed,precision=2,zerofill},
row style={3}{fixed,precision=2,zerofill},
row style={5}{fixed,precision=2,zerofill},
row style={7}{fixed,precision=2,zerofill},
row style={9}{fixed,precision=2,zerofill},
row style={11}{fixed,precision=2,zerofill},
every head row/.style={before row={\toprule}},
every row no 0/.style={before row=\midrule},
every row no 4/.style={before row=\midrule},
every row no 8/.style={before row=\midrule},
every last row/.style={after row={\bottomrule}},
]\tableFgOnetranspose
\caption{\label{tab:resultsA_mu100}Errors of the modified gradient-robust scheme $(\vecb{u}_h^+,\rho^+)$ and the classical scheme $(\vecb{u}_h,\rho)$ for \(\gamma=1\) and \(\mu=10^{-2}\).}
\end{sidewaystable}

The experiments want to answer the question, whether the
same locking behavior from the incompressible Stokes problem
can be observed in the compressible setting for the unmodified scheme.

Table~\ref{tab:resultsA_mu1} displays the calculated errors for \(\gamma=1\), \(\mu=1\)
and \(c=1\) or \(c=100\). Figure~\ref{fig:resultsA_mu1} shows the
corresponding convergence histories and convergence rates. In the compressible case (\(c=1\)), the classical and
the modified method both give very similar results. Interestingly, the convergence rate of the \(L^2\) velocity error drops asymptotically to the suboptimal rate \(1\) with respect to the mesh size \(h \approx \text{ndof}^{-1/2}\).
In the nearly incompressible case ($c=100$) however, all convergence rates are optimal as one would expect from the linear incompressible Stokes problem.
A similar trend can be observed for $\mu=10^{-1}$ in Figure~\ref{fig:resultsA_mu10}. Here, also the convergence rate of $\|\nabla(\vecb{u} - \vecb{u}^+_h)\|_{L^2}$ is asymptotically significantly suboptimal (but above $0.33$). The unmodified method
begins to show a similar behavior a bit later, possibly due to the pressure-dependent consistency error that dominates in the beginning but is reduced with optimal order.

To study the locking behavior, Figures~\ref{fig:resultsA_mu100} and \ref{fig:resultsA_mu10000} show the convergence histories of the calculated errors for \(\mu=10^{-2}\) and \(\mu=10^{-4}\), respectively, with \(c \in \lbrace 1, 100 \rbrace\).
The results for \(\mu=10^{-2}\) are also printed in Table~\ref{tab:resultsA_mu100}.
The first important observation is that the classical scheme $(\vecb{u}_h,\varrho_h)$ indeed shows locking and produces errors
that are several magnitudes larger than the errors of the modified scheme $(\vecb{u}^+_h,\varrho^+_h)$. The factor on coarse meshes is approximately \(1/\mu\) as expected by the theory. However, for the case \(c=1\) on finer meshes the velocity error convergence rates of the modified scheme deteriorates earlier than the ones of the classical scheme. Nevertheless, the error of the modified scheme on the finest mesh is still much smaller than the error of the classical scheme and it is expected that the classical scheme
 also shows suboptimal convergence once it arrives at the same error level similar to the case \(\mu=10^{-1}\).
Note, that for the nearly incompressible case \(c=100\), all convergence rates are again optimal and the
gap between $(\vecb{u}_h,\varrho_h)$ and $(\vecb{u}^+_h,\varrho^+_h)$ due to the locking is as large as in the other case and, more importantly, persists
even on the finest mesh. This is the known locking behavior from the incompressible Stokes setting.

% macros for pgfplotstable
\pgfplotstableset{
    % #1 = row index
    % #2 = row style keys
    row style/.style 2 args={
        every row #1 column 1/.style={#2},
        every row #1 column 2/.style={#2},
        every row #1 column 3/.style={#2},
        every row #1 column 4/.style={#2},
        every row #1 column 5/.style={#2},
        every row #1 column 6/.style={#2},
        every row #1 column 7/.style={#2},
        every row #1 column 8/.style={#2},
        every row #1 column 9/.style={#2},
        every row #1 column 10/.style={#2},
    }
}

%%%%%%%%%%%%%%%%%%%
% nu=1 c={1,100} %
%%%%%%%%%%%%%%%%%%%

% run: stokes_compressible_test.lua gamma=1.4 c=1 -symgrad nu=1 levels=8 fem=br -ut dt=0.5
% run: stokes_compressible_test.lua gamma=1.4 c=1 -symgrad nu=1 levels=8 fem=brlrrt -ut dt=0.5
% run: stokes_compressible_test.lua gamma=1.4 c=100 -symgrad nu=1 levels=8 fem=br -ut dt=0.01
% run: stokes_compressible_test.lua gamma=1.4 c=100 -symgrad nu=1 levels=8 fem=brlrrt -ut dt=0.01

\pgfplotstableread[col sep=ampersand,row sep=\\]{
ndof & L2uBR+ &	 H1uBR+ & L2rhoBR+ & L2uBR &	 H1uBR & L2rhoBR \\
     161 & 1.2082e-03    & 5.4065e-02    & 6.3120e-02    & 1.9269e-03    & 5.3870e-02    & 6.5116e-02 \\
     617 & 4.0813e-04    & 2.7731e-02    & 2.9839e-02    & 6.2364e-04    & 2.9080e-02    & 3.0374e-02 \\
    2297 & 1.4683e-04    & 1.3630e-02    & 1.5122e-02    & 2.0207e-04    & 1.4096e-02    & 1.5321e-02 \\
    9152 & 6.4183e-05    & 7.1482e-03    & 7.7596e-03    & 7.8111e-05    & 7.2210e-03    & 7.8279e-03 \\
   36326 & 3.0468e-05    & 3.6606e-03    & 3.9320e-03    & 3.3765e-05    & 3.6962e-03    & 3.9550e-03 \\
  143945 & 1.4883e-05    & 1.8727e-03    & 2.0171e-03    & 1.5643e-05    & 1.9041e-03    & 2.0243e-03 \\
  573386 & 7.3092e-06    & 9.6475e-04    & 1.0320e-03    & 7.4948e-06    & 9.8005e-04    & 1.0339e-03 \\
 2290184 & 3.6278e-06    & 4.9842e-04    & 5.2880e-04    & 3.6733e-06    & 5.0544e-04    & 5.2941e-04 \\
}\tableALaplace % nu=1 c=1

\pgfplotstableread[col sep=ampersand,row sep=\\]{
ndof & L2uBR+ &	 H1uBR+ & L2rhoBR+ & L2uBR &	 H1uBR & L2rhoBR \\
     161 & 8.6442e-04    & 4.4191e-02    & 6.2690e-04    & 1.5721e-03    & 4.5656e-02    & 6.4409e-04 \\
     617 & 2.1643e-04    & 2.1769e-02    & 2.9531e-04    & 4.2097e-04    & 2.3960e-02    & 3.0026e-04 \\
    2297 & 6.0194e-05    & 1.1079e-02    & 1.4831e-04    & 1.1433e-04    & 1.2051e-02    & 1.5055e-04 \\
    9152 & 1.4384e-05    & 5.4404e-03    & 7.5351e-05    & 2.8489e-05    & 5.8709e-03    & 7.6470e-05 \\
   36326 & 3.6824e-06    & 2.7316e-03    & 3.7577e-05    & 7.3144e-06    & 2.9517e-03    & 3.8068e-05 \\
  143945 & 9.2259e-07    & 1.3620e-03    & 1.8871e-05    & 1.8600e-06    & 1.4805e-03    & 1.9099e-05 \\
  573386 & 2.3475e-07    & 6.8028e-04    & 9.4254e-06    & 4.8238e-07    & 7.3940e-04    & 9.5345e-06 \\
 2290184 & 6.5525e-08    & 3.4190e-04    & 4.7155e-06    & 1.3069e-07    & 3.7105e-04    & 4.7674e-06 \\
}\tableBLaplace % nu=1 c=100

\pgfplotstableread[col sep=ampersand,row sep=\\]{
ndof & L2uBR+ &	 H1uBR+ & L2rhoBR+ & L2uBR &	 H1uBR & L2rhoBR \\
     161 & 1.6429e-03    & 5.3073e-02    & 6.2969e-02    & 2.3637e-03    & 5.2485e-02    & 6.5804e-02 \\
     617 & 5.1756e-04    & 2.7189e-02    & 2.9634e-02    & 7.3306e-04    & 2.7851e-02    & 3.0369e-02 \\
    2297 & 1.7347e-04    & 1.3641e-02    & 1.4928e-02    & 2.3427e-04    & 1.3895e-02    & 1.5211e-02 \\
    9152 & 6.8741e-05    & 7.3125e-03    & 7.6388e-03    & 8.3936e-05    & 7.2857e-03    & 7.7530e-03 \\
   36326 & 3.1697e-05    & 3.9700e-03    & 3.8534e-03    & 3.5254e-05    & 3.9355e-03    & 3.8967e-03 \\
  143945 & 1.5299e-05    & 2.2214e-03    & 1.9698e-03    & 1.6099e-05    & 2.2069e-03    & 1.9856e-03 \\
  573386 & 7.4367e-06    & 1.2497e-03    & 1.0084e-03    & 7.6331e-06    & 1.2456e-03    & 1.0136e-03 \\
 2290184 & 3.6650e-06    & 6.9508e-04    & 5.1834e-04    & 3.7132e-06    & 6.9268e-04    & 5.2018e-04 \\
}\tableA %\tableAsymGrad % nu=1 c=1 OK

\pgfplotstableread[col sep=ampersand,row sep=\\]{
ndof & L2uBR+ &	 H1uBR+ & L2rhoBR+ & L2uBR &	 H1uBR & L2rhoBR \\
     161 &  1.4263e-03    & 4.2681e-02    & 6.2661e-04    & 2.0767e-03    & 4.3959e-02    & 6.4982e-04 \\
     617 &  3.3656e-04    & 2.1185e-02    & 2.9504e-04    & 5.2617e-04    & 2.2722e-02    & 3.0132e-04 \\
    2297 &  9.3756e-05    & 1.0842e-02    & 1.4818e-04    & 1.4963e-04    & 1.1692e-02    & 1.5077e-04 \\
    9152 &  2.0188e-05    & 5.3258e-03    & 7.5321e-05    & 3.4998e-05    & 5.7093e-03    & 7.6607e-05 \\
   36326 &  5.2150e-06    & 2.6663e-03    & 3.7571e-05    & 8.9574e-06    & 2.8614e-03    & 3.8124e-05 \\
  143945 &  1.2726e-06    & 1.3287e-03    & 1.8868e-05    & 2.2227e-06    & 1.4336e-03    & 1.9125e-05 \\
  573386 &  3.2006e-07    & 6.6380e-04    & 9.4218e-06    & 5.6830e-07    & 7.1617e-04    & 9.5470e-06 \\
 2290184 &  8.5526e-08    & 3.3392e-04    & 4.7099e-06    & 1.5080e-07    & 3.5954e-04    & 4.7716e-06 \\
}\tableB %\tableBsymGrad % nu=1 c=100 OK

\pgfplotstabletranspose[string type,
	create on use/Lrate+/.style={create col/dyadic refinement rate={L2uBR+}},
	create on use/Lrate/.style={create col/dyadic refinement rate={L2uBR}},
	create on use/Hrate+/.style={create col/dyadic refinement rate={H1uBR+}},
	create on use/Hrate/.style={create col/dyadic refinement rate={H1uBR}},
	create on use/Drate+/.style={create col/dyadic refinement rate={L2rhoBR+}},
	create on use/Drate/.style={create col/dyadic refinement rate={L2rhoBR}},
    columns={ndof,L2uBR+,Lrate+,L2uBR,Lrate,H1uBR+,Hrate+,H1uBR,Hrate,L2rhoBR+,Drate+,L2rhoBR,Drate},
    colnames from=ndof,
    input colnames to=ndof
]\tableAtranspose{\tableA}

\pgfplotstabletranspose[string type,
	create on use/Lrate+/.style={create col/dyadic refinement rate={L2uBR+}},
	create on use/Lrate/.style={create col/dyadic refinement rate={L2uBR}},
	create on use/Hrate+/.style={create col/dyadic refinement rate={H1uBR+}},
	create on use/Hrate/.style={create col/dyadic refinement rate={H1uBR}},
	create on use/Drate+/.style={create col/dyadic refinement rate={L2rhoBR+}},
	create on use/Drate/.style={create col/dyadic refinement rate={L2rhoBR}},
    columns={ndof,L2uBR+,Lrate+,L2uBR,Lrate,H1uBR+,Hrate+,H1uBR,Hrate,L2rhoBR+,Drate+,L2rhoBR,Drate},
    colnames from=ndof,
    input colnames to=ndof
]\tableBtranspose{\tableB}

%%%%%%%%%%%%%%%%%%%%%%%%
% nu=10^{-1} c={1,100} %
%%%%%%%%%%%%%%%%%%%%%%%%

\pgfplotstableread[col sep=ampersand,row sep=\\]{
ndof & L2uBR+ &	 H1uBR+ & L2rhoBR+ & L2uBR &	 H1uBR & L2rhoBR \\
     161 & 1.2128e-03    & 5.4021e-02    & 6.2510e-02    & 1.2635e-02    & 1.9599e-01    & 6.4103e-02 \\
     617 & 4.1582e-04    & 2.7793e-02    & 2.9453e-02    & 3.6084e-03    & 1.1880e-01    & 2.9974e-02 \\
    2297 & 1.5683e-04    & 1.4087e-02    & 1.4795e-02    & 9.6923e-04    & 5.9314e-02    & 1.5012e-02 \\
    9152 & 6.9307e-05    & 7.8064e-03    & 7.5269e-03    & 2.6785e-04    & 3.1139e-02    & 7.6279e-03 \\
   36326 & 3.3052e-05    & 4.7538e-03    & 3.7638e-03    & 7.9242e-05    & 1.5723e-02    & 3.8063e-03 \\
  143945 & 1.6060e-05    & 3.1713e-03    & 1.9013e-03    & 2.6378e-05    & 8.2823e-03    & 1.9192e-03 \\
  573386 & 7.7759e-06    & 2.1674e-03    & 9.5994e-04    & 1.0068e-05    & 4.4059e-03    & 9.6679e-04 \\
 2290184 & 3.7916e-06    & 1.4340e-03    & 4.8793e-04    & 4.3116e-06    & 2.4116e-03    & 4.9042e-04 \\
}\tableCLaplace %nu=1e-1 c=1

\pgfplotstableread[col sep=ampersand,row sep=\\]{
ndof & L2uBR+ &	 H1uBR+ & L2rhoBR+ & L2uBR &	 H1uBR & L2rhoBR \\
     161 & 8.6442e-04    & 4.4192e-02    & 6.2502e-04    & 1.2314e-02    & 1.9282e-01    & 6.3937e-04 \\
     617 & 2.1643e-04    & 2.1769e-02    & 2.9447e-04    & 3.4397e-03    & 1.1699e-01    & 2.9956e-04 \\
    2297 & 6.0193e-05    & 1.1079e-02    & 1.4786e-04    & 8.9704e-04    & 5.9151e-02    & 1.5005e-04 \\
    9152 & 1.4382e-05    & 5.4403e-03    & 7.5160e-05    & 2.2965e-04    & 3.0773e-02    & 7.6240e-05 \\
   36326 & 3.6820e-06    & 2.7315e-03    & 3.7485e-05    & 5.8780e-05    & 1.5400e-02    & 3.7979e-05 \\
  143945 & 9.2245e-07    & 1.3620e-03    & 1.8826e-05    & 1.4857e-05    & 7.8251e-03    & 1.9064e-05 \\
  573386 & 2.3482e-07    & 6.8032e-04    & 9.4008e-06    & 3.7401e-06    & 3.9033e-03    & 9.5176e-06 \\
 2290184 & 6.5791e-08    & 3.4228e-04    & 4.6988e-06    & 9.4671e-07    & 1.9597e-03    & 4.7565e-06 \\
}\tableDLaplace %nu=1e-1 c=100

\pgfplotstableread[col sep=ampersand,row sep=\\]{
ndof & L2uBR+ &	 H1uBR+ & L2rhoBR+ & L2uBR &	 H1uBR & L2rhoBR \\
     161 & 1.6497e-03    & 5.3034e-02    & 6.2507e-02  & 1.1878e-02    & 1.8153e-01    & 6.4675e-02 \\
     617 & 5.2159e-04    & 2.7223e-02    & 2.9449e-02  & 3.3246e-03    & 1.0583e-01    & 3.0120e-02 \\
    2297 & 1.7897e-04    & 1.4021e-02    & 1.4787e-02  & 9.1819e-04    & 5.5487e-02    & 1.5046e-02 \\
    9152 & 7.2451e-05    & 7.9356e-03    & 7.5185e-03  & 2.5641e-04    & 2.9105e-02    & 7.6418e-03 \\
   36326 & 3.4049e-05    & 5.2936e-03    & 3.7523e-03  & 7.6839e-05    & 1.4910e-02    & 3.8065e-03 \\
  143945 & 1.6616e-05    & 4.0835e-03    & 1.8882e-03  & 2.6159e-05    & 8.1567e-03    & 1.9130e-03 \\
  573386 & 8.0728e-06    & 3.2522e-03    & 9.4773e-04  & 1.0213e-05    & 4.7709e-03    & 9.5882e-04 \\
 2290184 & 3.9269e-06    & 2.4560e-03    & 4.7871e-04  & 4.4230e-06    & 3.0112e-03    & 4.8346e-04 \\
}\tableC %\tableCSymGrad %nu=1e-1 c=1 OK

\pgfplotstableread[col sep=ampersand,row sep=\\]{
ndof & L2uBR+ &	 H1uBR+ & L2rhoBR+ & L2uBR &	 H1uBR & L2rhoBR \\
     161 & 1.4264e-03    & 4.2681e-02    & 6.2502e-04   & 1.1617e-02    & 1.7927e-01    & 6.4491e-04 \\
     617 & 3.3656e-04    & 2.1185e-02    & 2.9447e-04   & 3.1659e-03    & 1.0462e-01    & 3.0097e-04 \\
    2297 & 9.3757e-05    & 1.0842e-02    & 1.4785e-04   & 8.4427e-04    & 5.5351e-02    & 1.5041e-04 \\
    9152 & 2.0188e-05    & 5.3256e-03    & 7.5160e-05   & 2.1667e-04    & 2.8738e-02    & 7.6399e-05 \\
   36326 & 5.2150e-06    & 2.6661e-03    & 3.7484e-05   & 5.5320e-05    & 1.4419e-02    & 3.8044e-05 \\
  143945 & 1.2726e-06    & 1.3285e-03    & 1.8826e-05   & 1.3974e-05    & 7.3379e-03    & 1.9094e-05 \\
  573386 & 3.2009e-07    & 6.6366e-04    & 9.4008e-06   & 3.5197e-06    & 3.6600e-03    & 9.5323e-06 \\
 2290184 & 8.5605e-08    & 3.3392e-04    & 4.6987e-06   & 8.9167e-07    & 1.8374e-03    & 4.7639e-06 \\
}\tableD %\tableDSymGrad %nu=1e-1 c=100 OK

\pgfplotstabletranspose[string type,
	create on use/Lrate+/.style={create col/dyadic refinement rate={L2uBR+}},
	create on use/Lrate/.style={create col/dyadic refinement rate={L2uBR}},
	create on use/Hrate+/.style={create col/dyadic refinement rate={H1uBR+}},
	create on use/Hrate/.style={create col/dyadic refinement rate={H1uBR}},
	create on use/Drate+/.style={create col/dyadic refinement rate={L2rhoBR+}},
	create on use/Drate/.style={create col/dyadic refinement rate={L2rhoBR}},
    columns={ndof,L2uBR+,Lrate+,L2uBR,Lrate,H1uBR+,Hrate+,H1uBR,Hrate,L2rhoBR+,Drate+,L2rhoBR,Drate},
    colnames from=ndof,
    input colnames to=ndof
]\tableCtranspose{\tableC}

\pgfplotstabletranspose[string type,
	create on use/Lrate+/.style={create col/dyadic refinement rate={L2uBR+}},
	create on use/Lrate/.style={create col/dyadic refinement rate={L2uBR}},
	create on use/Hrate+/.style={create col/dyadic refinement rate={H1uBR+}},
	create on use/Hrate/.style={create col/dyadic refinement rate={H1uBR}},
	create on use/Drate+/.style={create col/dyadic refinement rate={L2rhoBR+}},
	create on use/Drate/.style={create col/dyadic refinement rate={L2rhoBR}},
    columns={ndof,L2uBR+,Lrate+,L2uBR,Lrate,H1uBR+,Hrate+,H1uBR,Hrate,L2rhoBR+,Drate+,L2rhoBR,Drate},
    colnames from=ndof,
    input colnames to=ndof
]\tableDtranspose{\tableD}

%%%%%%%%%%%%%%%%%%%%%%%%
% nu=10^{-2} c={1,100} %
%%%%%%%%%%%%%%%%%%%%%%%%

\pgfplotstableread[col sep=ampersand,row sep=\\]{
ndof & L2uBR+ &	 H1uBR+ & L2rhoBR+ & L2uBR &	 H1uBR & L2rhoBR \\
     161 & 1.6506e-03    & 5.3032e-02    & 6.2501e-02 & 1.1784e-01    & 1.7533e+00    & 6.4601e-02 \\
     617 & 5.2210e-04    & 2.7231e-02    & 2.9447e-02 & 3.2129e-02    & 1.0254e+00    & 3.0115e-02 \\
    2297 & 1.7975e-04    & 1.4102e-02    & 1.4785e-02 & 8.4472e-03    & 5.3937e-01    & 1.5043e-02 \\
    9152 & 7.3126e-05    & 8.0954e-03    & 7.5158e-03 & 2.1775e-03    & 2.8330e-01    & 7.6393e-03 \\
   36326 & 3.4673e-05    & 5.7646e-03    & 3.7484e-03 & 5.6166e-04    & 1.4128e-01    & 3.8042e-03 \\
  143945 & 1.7145e-05    & 5.0026e-03    & 1.8826e-03 & 1.4670e-04    & 7.2165e-02    & 1.9093e-03 \\
  573386 & 8.4871e-06    & 4.7274e-03    & 9.4023e-04 & 3.9269e-05    & 3.6179e-02    & 9.5321e-04 \\
 2290184 & 4.2037e-06    & 4.4438e-03    & 4.7016e-04 & 1.1266e-05    & 1.8505e-02    & 4.7649e-04 \\
}\tableE %\tableESymGrad %nu=1e-2 c=1 OK

\pgfplotstableread[col sep=ampersand,row sep=\\]{
ndof & L2uBR+ &	 H1uBR+ & L2rhoBR+ & L2uBR &	 H1uBR & L2rhoBR \\
     161 & 1.4264e-03    & 4.2681e-02    & 6.2500e-04 & 1.1603e-01    & 1.7508e+00    & 6.4457e-04\\
     617 & 3.3656e-04    & 2.1185e-02    & 2.9446e-04 & 3.1581e-02    & 1.0238e+00    & 3.0099e-04\\
    2297 & 9.3757e-05    & 1.0842e-02    & 1.4785e-04 & 8.3184e-03    & 5.4242e-01    & 1.5040e-04\\
    9152 & 2.0188e-05    & 5.3256e-03    & 7.5158e-05 & 2.1405e-03    & 2.8353e-01    & 7.6391e-05\\
   36326 & 5.2150e-06    & 2.6661e-03    & 3.7484e-05 & 5.4698e-04    & 1.4215e-01    & 3.8043e-05\\
  143945 & 1.2726e-06    & 1.3285e-03    & 1.8826e-05 & 1.3817e-04    & 7.2354e-02    & 1.9093e-05\\
  573386 & 3.2009e-07    & 6.6364e-04    & 9.4006e-06 & 3.4628e-05    & 3.6073e-02    & 9.5324e-06\\
 2290184 & 8.5615e-08    & 3.3392e-04    & 4.6986e-06 & 8.6884e-06    & 1.8115e-02    & 4.7639e-06\\
}\tableF %\tableFSymGrad %nu=1e-2 c=100 OK

\pgfplotstabletranspose[string type,
	create on use/Lrate+/.style={create col/dyadic refinement rate={L2uBR+}},
	create on use/Lrate/.style={create col/dyadic refinement rate={L2uBR}},
	create on use/Hrate+/.style={create col/dyadic refinement rate={H1uBR+}},
	create on use/Hrate/.style={create col/dyadic refinement rate={H1uBR}},
	create on use/Drate+/.style={create col/dyadic refinement rate={L2rhoBR+}},
	create on use/Drate/.style={create col/dyadic refinement rate={L2rhoBR}},
    columns={ndof,L2uBR+,Lrate+,L2uBR,Lrate,H1uBR+,Hrate+,H1uBR,Hrate,L2rhoBR+,Drate+,L2rhoBR,Drate},
    colnames from=ndof,
    input colnames to=ndof
]\tableEtranspose{\tableE}

\pgfplotstabletranspose[string type,
	create on use/Lrate+/.style={create col/dyadic refinement rate={L2uBR+}},
	create on use/Lrate/.style={create col/dyadic refinement rate={L2uBR}},
	create on use/Hrate+/.style={create col/dyadic refinement rate={H1uBR+}},
	create on use/Hrate/.style={create col/dyadic refinement rate={H1uBR}},
	create on use/Drate+/.style={create col/dyadic refinement rate={L2rhoBR+}},
	create on use/Drate/.style={create col/dyadic refinement rate={L2rhoBR}},
    columns={ndof,L2uBR+,Lrate+,L2uBR,Lrate,H1uBR+,Hrate+,H1uBR,Hrate,L2rhoBR+,Drate+,L2rhoBR,Drate},
    colnames from=ndof,
    input colnames to=ndof
]\tableFtranspose{\tableF}

%%%%%%%%%%%%%%%%%%%%%%%%
% nu=10^{-4} c={1,100} %
%%%%%%%%%%%%%%%%%%%%%%%%

\pgfplotstableread[col sep=ampersand,row sep=\\]{
ndof & L2uBR+ &	 H1uBR+ & L2rhoBR+ & L2uBR &	 H1uBR & L2rhoBR \\
     161 & 1.6507e-03    & 5.3032e-02    & 6.2501e-02    & 1.1800e+01    & 1.7536e+02    & 6.4595e-02 \\
     617 & 5.2216e-04    & 2.7232e-02    & 2.9447e-02    & 3.2125e+00    & 1.0248e+02    & 3.0116e-02 \\
    2297 & 1.7984e-04    & 1.4112e-02    & 1.4785e-02    & 8.4094e-01    & 5.3917e+01    & 1.5042e-02 \\
    9152 & 7.3207e-05    & 8.1158e-03    & 7.5158e-03    & 2.1488e-01    & 2.8341e+01    & 7.6393e-03 \\
   36326 & 3.4755e-05    & 5.8302e-03    & 3.7484e-03    & 5.4551e-02    & 1.4127e+01    & 3.8043e-03 \\
  143945 & 1.7226e-05    & 5.1487e-03    & 1.8826e-03    & 1.3813e-02    & 7.2075e+00    & 1.9093e-03 \\
  573386 & 8.5660e-06    & 5.0221e-03    & 9.4006e-04    & 3.4634e-03    & 3.5945e+00    & 9.5324e-04 \\
 2290184 & 4.2764e-06    & 4.9892e-03    & 4.6986e-04    & 8.6860e-04    & 1.8040e+00    & 4.7639e-04 \\
}\tableG %nu=1e-4 c=1 OK

\pgfplotstableread[col sep=ampersand,row sep=\\]{
ndof & L2uBR+ &	 H1uBR+ & L2rhoBR+ & L2uBR &	 H1uBR & L2rhoBR \\
     161 & 1.4264e-03    & 4.2681e-02    & 6.2500e-04    & 1.1616e+01    & 1.7508e+02    & 6.4453e-04 \\
     617 & 3.3656e-04    & 2.1185e-02    & 2.9446e-04    & 3.1603e+00    & 1.0232e+02    & 3.0099e-04 \\
    2297 & 9.3757e-05    & 1.0842e-02    & 1.4785e-04    & 8.3137e-01    & 5.4215e+01    & 1.5040e-04 \\
    9152 & 2.0188e-05    & 5.3256e-03    & 7.5158e-05    & 2.1393e-01    & 2.8354e+01    & 7.6390e-05 \\
   36326 & 5.2150e-06    & 2.6661e-03    & 3.7484e-05    & 5.4673e-02    & 1.4214e+01    & 3.8043e-05 \\
  143945 & 1.2726e-06    & 1.3285e-03    & 1.8826e-05    & 1.3810e-02    & 7.2345e+00    & 1.9093e-05 \\
  573386 & 3.2013e-07    & 6.6364e-04    & 9.4006e-06    & 3.4591e-03    & 3.6067e+00    & 9.5324e-06 \\
 2290184 & 8.5674e-08    & 3.3392e-04    & 4.6986e-06    & 8.6698e-04    & 1.8113e+00    & 4.7640e-06 \\
}\tableH %nu=1e-4 c=100 OK

\pgfplotstabletranspose[string type,
	create on use/Lrate+/.style={create col/dyadic refinement rate={L2uBR+}},
	create on use/Lrate/.style={create col/dyadic refinement rate={L2uBR}},
	create on use/Hrate+/.style={create col/dyadic refinement rate={H1uBR+}},
	create on use/Hrate/.style={create col/dyadic refinement rate={H1uBR}},
	create on use/Drate+/.style={create col/dyadic refinement rate={L2rhoBR+}},
	create on use/Drate/.style={create col/dyadic refinement rate={L2rhoBR}},
    columns={ndof,L2uBR+,Lrate+,L2uBR,Lrate,H1uBR+,Hrate+,H1uBR,Hrate,L2rhoBR+,Drate+,L2rhoBR,Drate},
    colnames from=ndof,
    input colnames to=ndof
]\tableGtranspose{\tableG}

\pgfplotstabletranspose[string type,
	create on use/Lrate+/.style={create col/dyadic refinement rate={L2uBR+}},
	create on use/Lrate/.style={create col/dyadic refinement rate={L2uBR}},
	create on use/Hrate+/.style={create col/dyadic refinement rate={H1uBR+}},
	create on use/Hrate/.style={create col/dyadic refinement rate={H1uBR}},
	create on use/Drate+/.style={create col/dyadic refinement rate={L2rhoBR+}},
	create on use/Drate/.style={create col/dyadic refinement rate={L2rhoBR}},
    columns={ndof,L2uBR+,Lrate+,L2uBR,Lrate,H1uBR+,Hrate+,H1uBR,Hrate,L2rhoBR+,Drate+,L2rhoBR,Drate},
    colnames from=ndof,
    input colnames to=ndof
]\tableHtranspose{\tableH}

%%%%%%%%
%FIGURES
%%%%%%%%

\begin{figure}
\centering
\begin{tikzpicture}
\begin{axis}[
         ymin=1e-8,
         ymax=1e-1,
         xmin=1e2,
         xmax=1e7,
        xmode=log,
        ymode=log,
        xlabel near ticks,
        ylabel near ticks,
        xlabel=ndof,
        width=0.405\textwidth,
        height=175pt,
       title=,
       legend style={at={(1.1,0.0)},anchor=south west,legend cell align=left}]
\addplot[color=blue,mark=o] table[x index=0, y index=1]{\tableA};    
\addlegendentry{\tiny $\| \vecb{u} - \vecb{u}^+_h \|_{L^2}$}
\addplot[color=red,mark=o] table[x index=0, y index=4]{\tableA};   
\addlegendentry{\tiny $\| \vecb{u} - \vecb{u}_h \|_{L^2}$}

\addplot[color=blue,mark=*] table[x index=0, y index=2]{\tableA};    
\addlegendentry{\tiny $\| \nabla(\vecb{u} - \vecb{u}^+_h) \|_{L^2}$}
\addplot[color=red,mark=*] table[x index=0, y index=5]{\tableA};     
\addlegendentry{\tiny $\| \nabla(\vecb{u} - \vecb{u}_h) \|_{L^2}$}

\addplot[color=blue,mark=x] table[x index=0, y index=3]{\tableA};    
\addlegendentry{\tiny $\| \varrho - \varrho^+_h \|_{L^2}$}
\addplot[color=red,mark=x] table[x index=0, y index=6]{\tableA};     
\addlegendentry{\tiny $\| \varrho - \varrho_h \|_{L^2}$}

\addplot[domain=1e0:1e7, color=gray, dashed]{0.4/pow(x,0.5)};
\addlegendentry{\tiny $h$}
\addplot[domain=1e0:1e7, color=gray,dotted]{1e-1/pow(x,1)};
\addlegendentry{\tiny $h^2$}
\end{axis}
\end{tikzpicture}
\hfill
\begin{tikzpicture}
\begin{axis}[
         ymin=1e-8,
         ymax=1e-1,
         xmin=1e2,
         xmax=1e7,
        xmode=log,
        ymode=log,
        xlabel near ticks,
        ylabel near ticks,
        xlabel=ndof,
        width=0.405\textwidth,
        height=175pt,
       title=,]
\addplot[color=blue,mark=o] table[x index=0, y index=1]{\tableB}; 
\addplot[color=red,mark=o] table[x index=0, y index=4]{\tableB};  
  
\addplot[color=blue,mark=*] table[x index=0, y index=2]{\tableB}; 
\addplot[color=red,mark=*] table[x index=0, y index=5]{\tableB}; 
   
\addplot[color=blue,mark=x] table[x index=0, y index=3]{\tableB};  
\addplot[color=red,mark=x] table[x index=0, y index=6]{\tableB}; 

\addplot[domain=1e0:1e7, color=gray, dashed]{0.4/pow(x,0.5)};
\addplot[domain=1e0:1e7, color=gray,dotted]{1e-1/pow(x,1)};
\end{axis}
\end{tikzpicture}

\caption{\label{fig:resultsB_mu1}
Convergence histories for the modified method and classical method for $\gamma=1.4$ and $\mu = 1$ and $c=1$ (left), $c=100$ (right).}
\end{figure}

\begin{figure}
\centering
\begin{tikzpicture}
\begin{axis}[
         ymin=1e-8,
         ymax=1e3,
         xmin=1e2,
         xmax=1e7,
        xmode=log,
        ymode=log,
        xlabel near ticks,
        ylabel near ticks,
        xlabel=ndof,
        width=0.405\textwidth,
        height=175pt,
       title=,
       legend style={at={(1.1,0.0)},anchor=south west,legend cell align=left}]
\addplot[color=blue,mark=o] table[x index=0, y index=1]{\tableG};    
\addlegendentry{\tiny $\| \vecb{u} - \vecb{u}^+_h \|_{L^2}$}
\addplot[color=red,mark=o] table[x index=0, y index=4]{\tableG};   
\addlegendentry{\tiny $\| \vecb{u} - \vecb{u}_h \|_{L^2}$}

\addplot[color=blue,mark=*] table[x index=0, y index=2]{\tableG};    
\addlegendentry{\tiny $\| \nabla(\vecb{u} - \vecb{u}^+_h) \|_{L^2}$}
\addplot[color=red,mark=*] table[x index=0, y index=5]{\tableG};     
\addlegendentry{\tiny $\| \nabla(\vecb{u} - \vecb{u}_h) \|_{L^2}$}

\addplot[color=blue,mark=x] table[x index=0, y index=3]{\tableG};    
\addlegendentry{\tiny $\| \varrho - \varrho^+_h \|_{L^2}$}
\addplot[color=red,mark=x] table[x index=0, y index=6]{\tableG};     
\addlegendentry{\tiny $\| \varrho - \varrho_h \|_{L^2}$}

\addplot[domain=1e0:1e7, color=gray, dashed]{0.4/pow(x,0.5)};
\addlegendentry{\tiny $h$}
\addplot[domain=1e0:1e7, color=gray,dotted]{1e-1/pow(x,1)};
\addlegendentry{\tiny $h^2$}
\end{axis}
\end{tikzpicture}
\hfill
\begin{tikzpicture}
\begin{axis}[
         ymin=1e-8,
         ymax=1e3,
         xmin=1e2,
         xmax=1e7,
        xmode=log,
        ymode=log,
        xlabel near ticks,
        ylabel near ticks,
        xlabel=ndof,
        width=0.405\textwidth,
        height=175pt,
       title=,]
\addplot[color=blue,mark=o] table[x index=0, y index=1]{\tableH}; 
\addplot[color=red,mark=o] table[x index=0, y index=4]{\tableH};  
  
\addplot[color=blue,mark=*] table[x index=0, y index=2]{\tableH}; 
\addplot[color=red,mark=*] table[x index=0, y index=5]{\tableH}; 
   
\addplot[color=blue,mark=x] table[x index=0, y index=3]{\tableH};  
\addplot[color=red,mark=x] table[x index=0, y index=6]{\tableH}; 

\addplot[domain=1e0:1e7, color=gray, dashed]{0.4/pow(x,0.5)};
\addplot[domain=1e0:1e7, color=gray,dotted]{1e-1/pow(x,1)};
\end{axis}
\end{tikzpicture}

\caption{\label{fig:resultsB_mu10000}
Convergence histories for the modified method and classical method for $\gamma=1.4$ and $\mu = 10^{-4}$ and $c=1$ (left), $c=100$ (right).}
\end{figure}
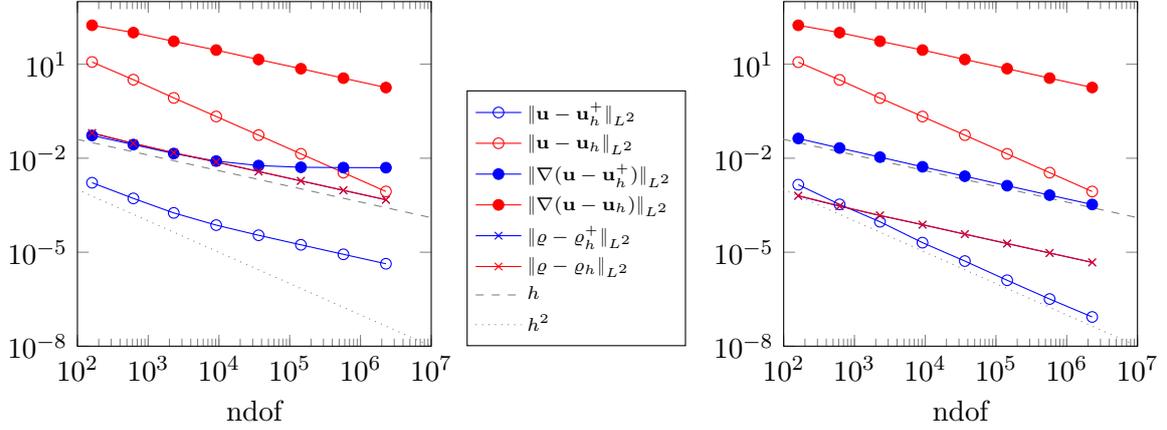

Figures~\ref{fig:resultsB_mu1} and \ref{fig:resultsB_mu10000} show some results for
\(\gamma = 1.4\) and \(\mu=1\) and \(\mu=10^{-4}\), respectively. The convergence histories
of the error are very similar to the isothermal case \(\gamma =1\), quantitatively and
qualitatively concerning the locking behavior and the suboptimal convergence rates for small \(mu\) and small \(c\).

\begin{remark}
The experiments convey that the suboptimal convergence rates on finer meshes have to do with the discretization
of the continuity equation and the compressibility of the fluid. If \(\mathrm{div}(\vecb{u}) \neq 0\), the upwind discretization introduces an error that does not allow any guaranteed convergence rates for the error
of the velocity gradient. However, due to an Aubin--Nitsche argument, the linear convergence of the \(L^2\) error
of the velocity is still granted and was observed in all experiments.
\end{remark}

%\begin{figure}
%\includegraphics[width = 0.244\textwidth, trim=20mm 15mm 35mm 55mm, clip]{plot_experiment2_mu2_br_lvl2.png}
%\includegraphics[width = 0.244\textwidth, trim=20mm 15mm 35mm 55mm, clip]{plot_experiment2_mu2_br_lvl3.png}
%\includegraphics[width = 0.244\textwidth, trim=20mm 15mm 35mm 55mm, clip]{plot_experiment2_mu2_br_lvl4.png}
%\includegraphics[width = 0.244\textwidth, trim=20mm 15mm 35mm 55mm, clip]{plot_experiment2_mu2_br_lvl5.png}\\
%\includegraphics[width = 0.244\textwidth, trim=20mm 15mm 35mm 55mm, clip]{plot_experiment2_mu2_brlrrt_lvl2.png}
%\includegraphics[width = 0.244\textwidth, trim=20mm 15mm 35mm 55mm, clip]{plot_experiment2_mu2_brlrrt_lvl3.png}
%\includegraphics[width = 0.244\textwidth, trim=20mm 15mm 35mm 55mm, clip]{plot_experiment2_mu2_brlrrt_lvl4.png}
%\includegraphics[width = 0.244\textwidth, trim=20mm 15mm 35mm 55mm, clip]{plot_experiment2_mu2_brlrrt_lvl5.png}
%
%\centering
%\hfill
%ndof=617
%\phantom{00000}
%\hfill
%ndof=2297
%\phantom{0000}
%\hfill
%ndof=9152
%\phantom{000}
%\hfill
%ndof=36326
%\phantom{0}
%\hfill
%\vspace{2ex}
%
%\caption{\label{fig:velocity_comparison_B}Velocity field and its absolute value in test case (B) and \(\mu=10^{-2}\) as computed by the classical scheme (top) and the modified scheme (bottom) on four consecutive mesh refinements with 617 to 36326
%degrees of freedom.}
%\end{figure}

\subsection{Incompressibility limit}\label{sec:example_incompressible_limit}
This example on the unit square \(\Omega := (0,1)^2\) examines the exact solution
\begin{align*}
  \vecb{u} := 0, \qquad p = \varphi(\varrho) := c \varrho^\gamma,
  \qquad \varrho(x,y) := 1.0 + (y-1/2)/c
\end{align*}
for \(\mu=1\) and \(\lambda = -2/3\).
These functions satisfy the compressible Stokes system with
the right-hand side functions
\begin{align*}
 \vecb{f} = 0 \quad \text{and} \quad \vecb{g} = \gamma \varrho^{\gamma-2}\begin{pmatrix}0\\1\end{pmatrix}.
 \end{align*}
Note that the constant \(c\) behaves like the
squared inverse of the Mach number and for \(c \to \infty\)
the compressible System converges to the incompressible Stokes system, i.e.,
the density converges to the constant function \(1.0\).

% ./build/stokes stokes_compressible_test3.lua fem=brlrrt nu=1 power=1 csteps=1 nref=1/2/3/4/5 -gravity
% ./build/stokes stokes_compressible_test3.lua fem=br nu=1 power=1 csteps=1 nref=1/2/3/4/5 -gravity
\pgfplotstableread[col sep=ampersand,row sep=\\]{
ndof & L2uBR+ &	 H1uBR+ & L2rhoBR+ & L2uBR &	 H1uBR & L2rhoBR \\
     161 & 2.2053e-05 &      3.7917e-04 &    6.2504e-02 & 8.9555e-04 &   1.3380e-02 &    6.3856e-02 \\
     617 & 1.0756e-05 &      2.0488e-04 &    2.9448e-02 & 2.4783e-04 &   8.1682e-03 &    2.9954e-02 \\
    2297 & 1.7069e-06 &      5.9351e-05 &    1.4785e-02 & 6.4319e-05 &   4.1439e-03 &    1.5004e-02 \\
    9152 & 3.1842e-07 &      1.5058e-05 &    7.5158e-03 & 1.6349e-05 &   2.1742e-03 &    7.6230e-03 \\
   36326 & 4.5439e-08 &      4.0321e-06 &    3.7484e-03 & 4.1878e-06 &   1.0860e-03 &    3.7978e-03 \\
  143945 & 6.1867e-09 &      1.0165e-06 &    1.8826e-03 & 1.0555e-06 &   5.5162e-04 &    1.9064e-03 \\
  573386 & 8.0138e-10 &      2.5576e-07 &    9.4006e-04 & 2.6432e-07 &   2.7507e-04 &    9.5180e-04 \\
}\tableAA

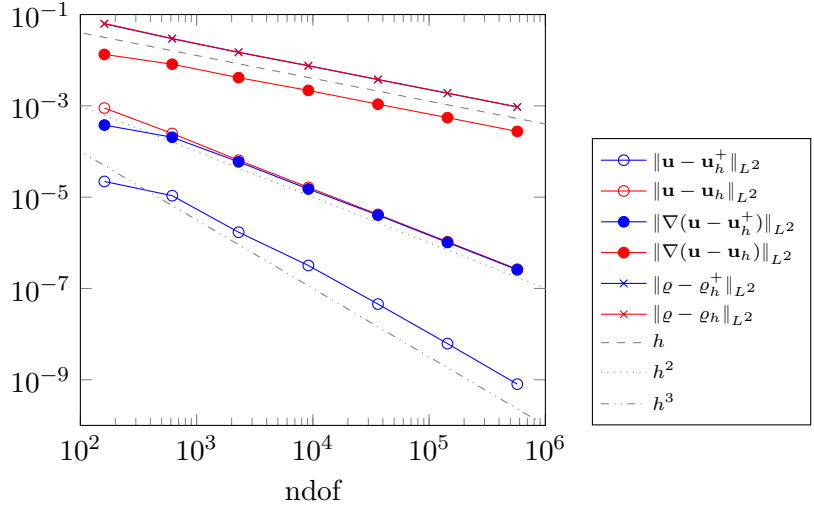
\begin{figure}
\centering
\begin{tikzpicture}
\begin{axis}[
         ymin=1e-10,
         ymax=1e-1,
         xmin=1e2,
         xmax=1e6,
        xmode=log,
        ymode=log,
        xlabel near ticks,
        ylabel near ticks,
        xlabel=ndof,
        width=0.5\textwidth,
        height=200pt,
       title=,
       legend style={at={(1.1,0.0)},anchor=south west,legend cell align=left}]
\addplot[color=blue,mark=o] table[x index=0, y index=1]{\tableAA};    
\addlegendentry{\tiny $\| \vecb{u} - \vecb{u}^+_h \|_{L^2}$}
\addplot[color=red,mark=o] table[x index=0, y index=4]{\tableAA};   
\addlegendentry{\tiny $\| \vecb{u} - \vecb{u}_h \|_{L^2}$}

\addplot[color=blue,mark=*] table[x index=0, y index=2]{\tableAA};    
\addlegendentry{\tiny $\| \nabla(\vecb{u} - \vecb{u}^+_h) \|_{L^2}$}
\addplot[color=red,mark=*] table[x index=0, y index=5]{\tableAA};     
\addlegendentry{\tiny $\| \nabla(\vecb{u} - \vecb{u}_h) \|_{L^2}$}

\addplot[color=blue,mark=x] table[x index=0, y index=3]{\tableAA};    
\addlegendentry{\tiny $\| \varrho - \varrho^+_h \|_{L^2}$}
\addplot[color=red,mark=x] table[x index=0, y index=6]{\tableAA};     
\addlegendentry{\tiny $\| \varrho - \varrho_h \|_{L^2}$}

\addplot[domain=1e0:1e7, color=gray, dashed]{0.4/pow(x,0.5)};
\addlegendentry{\tiny $h$}
\addplot[domain=1e0:1e7, color=gray,dotted]{1e-1/pow(x,1)};
\addlegendentry{\tiny $h^2$}
\addplot[domain=1e0:1e7, color=gray,dash dot dot]{1e-1/pow(x,1.5)};
\addlegendentry{\tiny $h^3$}
\end{axis}
\end{tikzpicture}
\caption{\label{fig:example2_c1_alpha1}Convergence histories for the modified gradient-robust scheme $(\vecb{u}_h^+,\rho^+)$ and the classical scheme $(\vecb{u}_h,\rho)$ for \(c=1\) and \(\gamma=1\) on unstructured meshes in Section~\ref{sec:example_incompressible_limit}.}
\end{figure}

\begin{table}
{\footnotesize
\pgfplotstabletypeset[columns={ndof,L2uBR+,H1uBR+,L2rhoBR+,L2uBR,H1uBR,L2rhoBR},
%every head row/.style={before row={ & \multicolumn{3}{c}{(p-robust scheme)}\\},
every head row/.style={before row={\toprule}, after row={\midrule}},
every last row/.style={after row={\bottomrule}},
columns/ndof/.style={int detect, column type={c}, column name=\textsc{ndof},set thousands separator={}},
columns/L2uBR+/.style={precision=4,sci,sci 10e,sci zerofill, column name=$\| \vecb{u} - \vecb{u}^+_h\|_{L^2} $},
columns/H1uBR+/.style={precision=4,sci,sci 10e,sci zerofill, column name=$\| \nabla(\vecb{u} - \vecb{u}^+_h)\|_{L^2} $},
columns/L2rhoBR+/.style={precision=4,sci,sci 10e,sci zerofill, column name=$\| \rho - \rho^+_h\|_{L^2} $},
columns/L2uBR/.style={precision=4,sci,sci 10e,sci zerofill, column name=$\| \vecb{u} - \vecb{u}_h\|_{L^2} $},
columns/H1uBR/.style={precision=4,sci,sci 10e,sci zerofill, column name=$\| \nabla(\vecb{u} - \vecb{u}_h)\|_{L^2} $},
columns/L2rhoBR/.style={precision=4,sci,sci 10e,sci zerofill, column name=$\| \rho - \rho_h\|_{L^2} $}
]\tableAA
}
\caption{\label{tab:example2_c1_alpha1}Errors of the modified gradient-robust scheme $(\vecb{u}_h^+,\rho^+)$ and the classical scheme $(\vecb{u}_h,\rho)$ for \(c=1\) and \(\gamma=1\) on unstructured meshes in Section~\ref{sec:example_incompressible_limit}.}
\end{table}

Table~\ref{tab:example2_c1_alpha1} compares the error of the solutions of the classical Bernardi--Raugel scheme (\(\Pi=1\)) with
the gradient-robust scheme (where \(\Pi\) is chosen as described above) for \(c=1\) and \(\gamma=1\). One can clearly see, that
the velocity errors of the novel scheme are improved by about two orders of magnitudes and also show some superconvergence behavior as depicted in Figure~\ref{fig:example2_c1_alpha1}, i.e., the \(L^2\) velocity gradient error converges quadratically.

\begin{figure}
\hfill
\includegraphics[width = 0.33\textwidth, trim=20mm 20mm 20mm 20mm, clip]{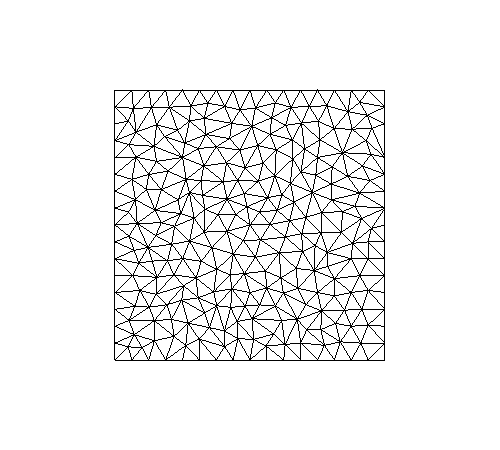}
\hfill
\includegraphics[width = 0.33\textwidth, trim=20mm 20mm 20mm 20mm, clip]{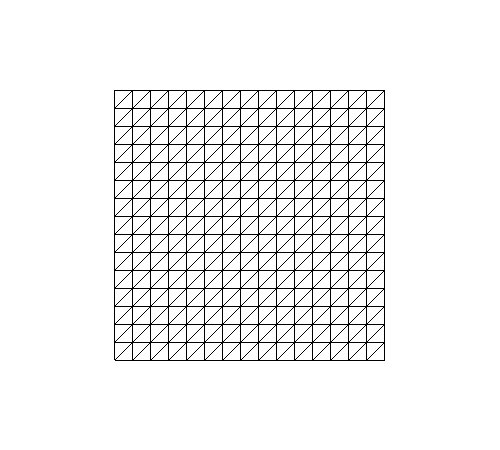}
\hfill
\hfill
\caption{\label{fig:meshes}Unstructured grid (G1) with \(489\) triangles (left) and structured mesh grid (G2) with \(450\) triangles (right) used in the examples from Section~\ref{sec:example_incompressible_limit} with varying $c$.}
\end{figure}

% ./build/stokes stokes_compressible_test3.lua fem=brlrrt nu=1 power=2 csteps=5 nref=3 -gravity
% ./build/stokes stokes_compressible_test3.lua fem=brlrrt nu=1 power=2 csteps=5 nref=3 -gravity -structured
% ./build/stokes stokes_compressible_test3.lua fem=br nu=1 power=2 csteps=5 nref=3 -gravity
% ./build/stokes stokes_compressible_test3.lua fem=br nu=1 power=2 csteps=5 nref=3 -gravity -structured
\pgfplotstableread[col sep=ampersand,row sep=\\]{
c & L2uBRunstruct+ &	L2uBRstruct+ & L2uBRunstruct & L2uBRstruct \\
     1 & 1.0219e-04 & 2.9412e-13 & 8.6144e-03 & 1.3473e-02 \\
    10 & 1.0236e-05 & 9.3288e-14 & 8.2877e-03 & 1.3001e-02 \\
   100 & 1.0236e-06 & 7.8271e-13 & 8.2862e-03 & 1.2996e-02 \\
  1000 & 1.0236e-07 & 8.3695e-12 & 8.2864e-03 & 1.2996e-02 \\
 10000 & 1.0236e-08 & 6.8088e-11 & 8.2864e-03 & 1.2996e-02 \\
}\tableBgA %gamma=2

% ./build/stokes stokes_compressible_test3.lua fem=brlrrt nu=1 power=1.4 csteps=5 nref=3 -gravity
% ./build/stokes stokes_compressible_test3.lua fem=brlrrt nu=1 power=1.4 csteps=5 nref=3 -gravity -structured
% ./build/stokes stokes_compressible_test3.lua fem=br nu=1 power=1.4 csteps=5 nref=3 -gravity
% ./build/stokes stokes_compressible_test3.lua fem=br nu=1 power=1.4 csteps=5 nref=3 -gravity -structured
\pgfplotstableread[col sep=ampersand,row sep=\\]{
c & L2uBRunstruct+ &	L2uBRstruct+ & L2uBRunstruct & L2uBRstruct \\
     1 & 7.6708e-05 & 1.8710e-06 & 5.7730e-03 & 9.0664e-03 \\
    10 & 7.1800e-06 & 1.3956e-08 & 5.7996e-03 & 9.0971e-03 \\
   100 & 7.1664e-07 & 1.3977e-10 & 5.8004e-03 & 9.0974e-03 \\
  1000 & 7.1653e-08 & 1.4607e-11 & 5.8005e-03 & 9.0974e-03 \\
 10000 & 7.1665e-09 & 1.0489e-10 & 5.8005e-03 & 9.0974e-03 \\
}\tableBgB %gamma=1.4

% ./build/stokes stokes_compressible_test3.lua fem=brlrrt nu=1 power=1 csteps=5 nref=3 -gravity
% ./build/stokes stokes_compressible_test3.lua fem=brlrrt nu=1 power=1 csteps=5 nref=3 -gravity -structured
% ./build/stokes stokes_compressible_test3.lua fem=br nu=1 power=1 csteps=5 nref=3 -gravity
% ./build/stokes stokes_compressible_test3.lua fem=br nu=1 power=1 csteps=5 nref=3 -gravity -structured
\pgfplotstableread[col sep=ampersand,row sep=\\]{
c & L2uBRunstruct+ &	L2uBRstruct+ & L2uBRunstruct & L2uBRstruct \\
     1 & 5.9299e-05 & 2.5549e-06 & 4.1439e-03 & 6.4992e-03 \\
    10 & 5.1372e-06 & 1.6635e-08 & 4.1432e-03 & 6.4981e-03 \\
   100 & 5.1194e-07 & 1.6628e-10 & 4.1432e-03 & 6.4982e-03 \\
  1000 & 5.1181e-08 & 1.5016e-11 & 4.1432e-03 & 6.4982e-03 \\
 10000 & 5.1275e-09 & 1.6017e-10 & 4.1432e-03 & 6.4982e-03 \\
}\tableBgC %gamma=1.0

\begin{table}
{\footnotesize
\pgfplotstabletypeset[columns={c,L2uBRunstruct+,L2uBRunstruct,L2uBRstruct+,L2uBRstruct},
%every head row/.style={before row={ & \multicolumn{3}{c}{(p-robust scheme)}\\},
every head row/.style={before row={\toprule}, after row={\midrule}},
every last row/.style={after row={\bottomrule}},
columns/c/.style={int detect, column type={c}, column name=$c$,set thousands separator={}},
columns/L2uBRunstruct+/.style={precision=4,sci,sci 10e,sci zerofill, column name=$\| \nabla(\vecb{u} - \vecb{u}^+_h)\|_{L^2} \text{ (G1)} $},
columns/L2uBRstruct+/.style={precision=4,sci,sci 10e,sci zerofill, column name=$\| \nabla(\vecb{u} - \vecb{u}^+_h)\|_{L^2} \text{ (G2)} $},
columns/L2uBRunstruct/.style={precision=4,sci,sci 10e,sci zerofill, column name=$\| \nabla(\vecb{u} - \vecb{u}_h)\|_{L^2} \text{ (G1)} $},
columns/L2uBRstruct/.style={precision=4,sci,sci 10e,sci zerofill, column name=$\| \nabla(\vecb{u} - \vecb{u}_h)\|_{L^2} \text{ (G2)} $}
]\tableBgA
}
\caption{\label{tab:example2_alpha2_meshcomparison}Errors \(\| \nabla(\vecb{u} - \vecb{u}_h) \|_{L^2}\) of the classical and gradient-robust scheme computed on the two fixed grids from Figure~\ref{fig:meshes}
for \(\gamma=2\) and different choices of \(c\)
in the example from Section~\ref{sec:example_incompressible_limit}.}
\end{table}

For the following discussion we fix two meshes, one is the unstructured mesh with \(489\) triangles used before and the other one is a structured mesh with \(450\) triangles, see Figure~\ref{fig:meshes}.
Table~\ref{tab:example2_alpha2_meshcomparison} compares the velocity error on these two meshes for different choices
of \(c\). There are two interesting observations. First, the velocity errors of the gradient-robust scheme converge to zero
for \(c \rightarrow \infty\), while the errors of the classical scheme stagnates. Second, the velocity of the gradient-robust scheme is exact on structured meshes for every \(c\), while the classical scheme is not.

\begin{table}
{\footnotesize
\pgfplotstabletypeset[columns={c,L2uBRunstruct+,L2uBRunstruct,L2uBRstruct+,L2uBRstruct},
%every head row/.style={before row={ & \multicolumn{3}{c}{(p-robust scheme)}\\},
every head row/.style={before row={\toprule}, after row={\midrule}},
every last row/.style={after row={\bottomrule}},
columns/c/.style={int detect, column type={c}, column name=$c$,set thousands separator={}},
columns/L2uBRunstruct+/.style={precision=4,sci,sci 10e,sci zerofill, column name=$\| \nabla(\vecb{u} - \vecb{u}^+_h)\|_{L^2} \text{ (G1)} $},
columns/L2uBRstruct+/.style={precision=4,sci,sci 10e,sci zerofill, column name=$\| \nabla(\vecb{u} - \vecb{u}^+_h)\|_{L^2} \text{ (G2)} $},
columns/L2uBRunstruct/.style={precision=4,sci,sci 10e,sci zerofill, column name=$\| \nabla(\vecb{u} - \vecb{u}_h)\|_{L^2} \text{ (G1)} $},
columns/L2uBRstruct/.style={precision=4,sci,sci 10e,sci zerofill, column name=$\| \nabla(\vecb{u} - \vecb{u}_h)\|_{L^2} \text{ (G2)} $}
]\tableBgB
}

\caption{\label{tab:example2_alpha14_meshcomparison}Errors \(\| \nabla(\vecb{u} - \vecb{u}_h) \|_{L^2}\) of the classical and gradient-robust scheme computed on the two fixed grids from Figure~\ref{fig:meshes}
for \(\gamma=1.4\) and different choices of \(c\)
in the example from Section~\ref{sec:example_incompressible_limit}.}
\end{table}

\begin{table}
{\footnotesize
\pgfplotstabletypeset[columns={c,L2uBRunstruct+,L2uBRunstruct,L2uBRstruct+,L2uBRstruct},
%every head row/.style={before row={ & \multicolumn{3}{c}{(p-robust scheme)}\\},
every head row/.style={before row={\toprule}, after row={\midrule}},
every last row/.style={after row={\bottomrule}},
columns/c/.style={int detect, column type={c}, column name=$c$,set thousands separator={}},
columns/L2uBRunstruct+/.style={precision=4,sci,sci 10e,sci zerofill, column name=$\| \nabla(\vecb{u} - \vecb{u}^+_h)\|_{L^2} \text{ (G1)} $},
columns/L2uBRstruct+/.style={precision=4,sci,sci 10e,sci zerofill, column name=$\| \nabla(\vecb{u} - \vecb{u}^+_h)\|_{L^2} \text{ (G2)} $},
columns/L2uBRunstruct/.style={precision=4,sci,sci 10e,sci zerofill, column name=$\| \nabla(\vecb{u} - \vecb{u}_h)\|_{L^2} \text{ (G1)} $},
columns/L2uBRstruct/.style={precision=4,sci,sci 10e,sci zerofill, column name=$\| \nabla(\vecb{u} - \vecb{u}_h)\|_{L^2} \text{ (G2)} $}
]\tableBgC
}
\caption{\label{tab:example2_alpha1_meshcomparison}Errors \(\| \nabla(\vecb{u} - \vecb{u}_h) \|_{L^2}\) of the classical and gradient-robust scheme computed on the two fixed grids from Figure~\ref{fig:meshes}
for \(\gamma=1\) and different choices of \(c\)
in the example from Section~\ref{sec:example_incompressible_limit}.}
\end{table}

Table~\ref{tab:example2_alpha14_meshcomparison} 
and \ref{tab:example2_alpha1_meshcomparison} repeat this experiment for
\(\gamma=1.4\) and \(\gamma = 1\), respectively. Here the results are similar as for the case with \(\gamma=2\)
in the sense that the gradient-robust scheme is more accurate than the classical scheme. However, the gradient-robust variant is not exact on structured meshes in these cases which most likely is due to the non-constant vector \(\vecb{g}\).

\subsection{Well-balanced property}\label{sec:example_wellbalanced}
We repeat the experiment from the previous section, but this time we assume the right-hand sides
\begin{align}\label{eqn:rhs_wellbalanced}
 \vecb{f} = \gamma \varrho^{\gamma-1} \begin{pmatrix}0\\1\end{pmatrix} \quad \text{and} \quad \vecb{g} = 0.
 \end{align}
 Table~\ref{tab:example_wb_gamma1} displays the errors for \(\gamma=1\) and \(c=1\) for the classical and the
 gradient-robust scheme. Surprisingly, the novel gradient-robust scheme computes the exact velocity even
 on unstructured meshes. Also note, that the gradient-robust scheme converges after the first iteration, since
 the initial value based on the (rescaled) discrete pressure from incompressible Stokes problem is already the
 correct discrete density. Table~\ref{tab:example_wb_gamma14} leads to
 the same conclusions for for \(\gamma = 1.4\).
 
% ./build/stokes stokes_compressible_test3.lua fem=brlrrt nu=1 power=1.4 csteps=1 nref=1/2/3/4/5
% ./build/stokes stokes_compressible_test3.lua fem=br nu=1 power=1.4 csteps=1 nref=1/2/3/4/5
\pgfplotstableread[col sep=ampersand,row sep=\\]{
ndof & L2uBR+ &	 H1uBR+ & L2rhoBR+ & L2uBR &	 H1uBR & L2rhoBR \\
  161   & 7.5658e-14 & 	 1.5323e-12 & 	 6.2501e-02 & 1.2578e-03 & 	 1.8979e-02 & 	 6.4038e-02 \\
  617   & 1.3088e-16 & 	 5.2005e-15 & 	 2.9447e-02 & 3.4992e-04 & 	 1.1486e-02 & 	 2.9970e-02 \\
  2297  & 8.0186e-17 & 	 5.5293e-15 & 	 1.4785e-02 & 9.0438e-05 & 	 5.7741e-03 & 	 1.5006e-02 \\
 9152   & 8.2615e-17 & 	 1.1330e-14 & 	 7.5158e-03 & 2.2972e-05 & 	 3.0452e-03 & 	 7.6233e-03 \\
 36326  & 8.4962e-17 & 	 2.2574e-14 & 	 3.7484e-03 & 5.8270e-06 & 	 1.5099e-03 & 	 3.7979e-03 \\
}\tableCC

% ./build/stokes stokes_compressible_test3.lua fem=brlrrt nu=1 power=1 csteps=1 nref=1/2/3/4/5
% ./build/stokes stokes_compressible_test3.lua fem=br nu=1 power=1 csteps=1 nref=1/2/3/4/5
\pgfplotstableread[col sep=ampersand,row sep=\\]{
ndof & L2uBR+ &	 H1uBR+ & L2rhoBR+ & L2uBR &	 H1uBR & L2rhoBR \\
  161   & 6.9935e-17 &   1.2646e-15 &    6.2500e-02 & 8.9980e-04 &   1.3467e-02 &    6.3956e-02 \\
  617   & 6.6351e-17 &   2.4263e-15 &    2.9446e-02 & 2.4970e-04 &   8.1662e-03 &    2.9961e-02 \\
  2297  & 7.3217e-17 &   5.0692e-15 &    1.4785e-02 & 6.4488e-05 &   4.1437e-03 &    1.5005e-02 \\
 9152   & 7.3142e-17 &   1.0417e-14 &    7.5158e-03 & 1.6409e-05 &   2.1739e-03 &    7.6232e-03 \\
 36326  & 7.6058e-17 &   2.0748e-14 &    3.7484e-03 & 4.1974e-06 &   1.0860e-03 &    3.7978e-03 \\
}\tableCD

\begin{table}
{\footnotesize
\pgfplotstabletypeset[columns={ndof,L2uBR+,H1uBR+,L2rhoBR+,L2uBR,H1uBR,L2rhoBR},
%every head row/.style={before row={ & \multicolumn{3}{c}{(p-robust scheme)}\\},
every head row/.style={before row={\toprule}, after row={\midrule}},
every last row/.style={after row={\bottomrule}},
columns/ndof/.style={int detect, column type={c}, column name=\textsc{ndof},set thousands separator={}},
columns/L2uBR+/.style={precision=4,sci,sci 10e,sci zerofill, column name=$\| \vecb{u} - \vecb{u}^+_h\|_{L^2} $},
columns/H1uBR+/.style={precision=4,sci,sci 10e,sci zerofill, column name=$\| \nabla(\vecb{u} - \vecb{u}^+_h)\|_{L^2} $},
columns/L2rhoBR+/.style={precision=4,sci,sci 10e,sci zerofill, column name=$\| \rho - \rho^+_h\|_{L^2} $},
columns/L2uBR/.style={precision=4,sci,sci 10e,sci zerofill, column name=$\| \vecb{u} - \vecb{u}_h\|_{L^2} $},
columns/H1uBR/.style={precision=4,sci,sci 10e,sci zerofill, column name=$\| \nabla(\vecb{u} - \vecb{u}_h)\|_{L^2} $},
columns/L2rhoBR/.style={precision=4,sci,sci 10e,sci zerofill, column name=$\| \rho - \rho_h\|_{L^2} $}
]\tableCC
}
\caption{\label{tab:example_wb_gamma14}Errors of the modified gradient-robust scheme $(\vecb{u}_h^+,\rho^+)$ and the classical scheme $(\vecb{u}_h,\rho)$ for \(c=1\) and \(\gamma=1.4\) on unstructured grids with right-hand sides \eqref{eqn:rhs_wellbalanced}.}
\end{table}

\begin{table}
{\footnotesize
\pgfplotstabletypeset[columns={ndof,L2uBR+,H1uBR+,L2rhoBR+,L2uBR,H1uBR,L2rhoBR},
%every head row/.style={before row={ & \multicolumn{3}{c}{(p-robust scheme)}\\},
every head row/.style={before row={\toprule}, after row={\midrule}},
every last row/.style={after row={\bottomrule}},
columns/ndof/.style={int detect, column type={c}, column name=\textsc{ndof},set thousands separator={}},
columns/L2uBR+/.style={precision=4,sci,sci 10e,sci zerofill, column name=$\| \vecb{u} - \vecb{u}^+_h\|_{L^2} $},
columns/H1uBR+/.style={precision=4,sci,sci 10e,sci zerofill, column name=$\| \nabla(\vecb{u} - \vecb{u}^+_h)\|_{L^2} $},
columns/L2rhoBR+/.style={precision=4,sci,sci 10e,sci zerofill, column name=$\| \rho - \rho^+_h\|_{L^2} $},
columns/L2uBR/.style={precision=4,sci,sci 10e,sci zerofill, column name=$\| \vecb{u} - \vecb{u}_h\|_{L^2} $},
columns/H1uBR/.style={precision=4,sci,sci 10e,sci zerofill, column name=$\| \nabla(\vecb{u} - \vecb{u}_h)\|_{L^2} $},
columns/L2rhoBR/.style={precision=4,sci,sci 10e,sci zerofill, column name=$\| \rho - \rho_h\|_{L^2} $}
]\tableCD
}
\caption{\label{tab:example_wb_gamma1}Errors of the modified gradient-robust scheme $(\vecb{u}_h^+,\rho^+)$ and the classical scheme $(\vecb{u}_h,\rho)$ for \(c=1\) and \(\gamma=1\) on unstructured grids with right-hand sides \eqref{eqn:rhs_wellbalanced}.}
\end{table}

%-- BIBLIOGRAPHY ------------------------------------------
\bibliographystyle{amsplain}
\bibliography{references}

\end{document}